\documentclass[11pt]{article}

\usepackage{a4wide,latexsym,amssymb,amsthm,color}
\usepackage[centertags]{amsmath}
\usepackage{tikz-cd}
\theoremstyle{plain}
\usepackage{ulem}
\usepackage{enumerate}
\usepackage{authblk}

\newtheorem{theorem}{Theorem}
\newtheorem{lemma}{Lemma}
\newtheorem{proposition}{Proposition}
\newtheorem{corollary}{Corollary}
\newtheorem{definition}{Definition}
\newtheorem{remark}{Remark}
\newtheorem{example}{Example}
\theoremstyle{remark}

\newcommand\bigzero{\makebox(0,0){\text{\huge0}}}
\newcommand{\sn}{\mathbb{S}^{2n+1}_{2p}}
\newcommand{\cn}{\mathbb{C}^{n+1}_p}
\newcommand{\cpn}{\mathbb{C}P^n_p}
\newcommand{\snp}{\mathbb{S}^{2n+1}_{2p}}
\newcommand{\RR}{\mathbb{R}}

\newcommand{\CC}{\mathbb{C}}
\newcommand{\dd}{\mathcal{D}}
\newcommand{\ep}{\varepsilon}
\newcommand{\de}{(-\delta,\delta)}
\newcommand{\hp}{\hat{\psi}}

\title{{Ruled Real Hypersurfaces in the\\ Indefinite Complex Projective Space}}

\author[1]{Marilena Moruz}
\affil[1]{Al.I. Cuza University of Iaşi, Faculty of Mathematics \authorcr Bd. Carol I, n. 11, 700506 Ia\c si, Romania\authorcr marilena.moruz@gmail.com\authorcr \, }

\author[2]{Miguel Ortega}
\author[2]{Juan de Dios P\'erez\thanks{First author is supported by a grant of the Romanian Ministry of Education and Research, CNCS - UEFISCDI, project number PN-III-P1-1.1-PD-2019-0253, within PNCDI III. Ortega and P\'erez partially financed by: (1) the Spanish MICINN and ERDF, project PID2020-116126GB-I00; (2) the “Maria de Maeztu” Excellence Unit IMAG,
ref. CEX2020-001105-M, funded by MCIN/AEI/10.13039/501100011033/. Ortega also partially supported by the Andaluc\'ia grant A-FQM-494-UGR18. Pérez also partially supported by the Spanish FEDER-Andaluc\'ia grant PY20 01391.}}

\affil[2]{Department of Geometry and Topology, Faculty of Sciences,
Institute of Mathematics IMAG, Universidad de Granada, 18071 Granada
Spain \authorcr miortega@ugr.es, jdperez@ugr.es}

\begin{document}

\maketitle

\begin{abstract} The main two families of real hypersurfaces in complex space forms are Hopf and ruled. However, very little is known about real hypersurfaces in the indefinite complex projective space $\cpn$. In a previous work, Kimura and the second author introduced Hopf real hypersurfaces in $\cpn$. In this paper, ruled real hypersurfaces in the indefinite complex projective space are introduced, as those whose maximal holomorphic distribution is integrable, and such that the leaves are totally geodesic holomorphic hyperplanes. A detailed description of the shape operator is computed, obtaining two main different families. A method of construction is exhibited, by gluing in a suitable way totally geodesic holomorphic hyperplanes along a non-null curve.  Next, the classification of all minimal ruled real hypersurfaces  is obtained, in terms of three main families of curves, namely geodesics, totally real circles and a third case which is not a Frenet curve, but can be explicitly computed. Four examples of minimal ruled real hypersurfaces are described. 
\end{abstract}

\noindent \textbf{Keywords:} Real hypersurface, indefinite complex projective space, ruled real hypersurface.\\
\noindent \textbf{2010MSC:}  
Primary 
53B25, 
53C50; 
Secondary 53C42, 
53B30. 

\section{Introduction}

The study of real hypersurfaces  in non-flat complex space forms has been strongly developed during the last 50 years. Since a short list would not be fair, a good survey 
can be found in Cecil and Ryan's book \cite{CR}. Anyway, the main pillars of the theory of real hypersurfaces in complex projective space $\mathbb{C}P^n$ are the following papers. First, Takagi, \cite{T1}, classified the homogeneous ones, obtaining a list of 6 examples. Second, Cecil and Ryan developed the focal theory in \cite{CR1}, obtaining local characterizations of Hopf examples, namely, those whose Reeb vector field $\xi=-JN$ is principal, where $N$ is a unit normal vector field and $J$ is the complex structure of $\mathbb{C}P^n$. Next, Kimura \cite{K2} characterized the Takagi's examples as the only ones which are Hopf and all of whose principal curvatures are constant. Similar studies for the complex hyperbolic space were made, but we recommend again \cite{CR}. 

In addition,  Kimura also introduced in \cite{K1} a very important family of real hypersurfaces in the complex projective space, the so-called \textit{ruled} ones. Indeed, if $\cal D$ denotes the \textit{maximal holomorphic} distribution on $M$, given at each point $x \in M$ by all the vectors orthogonal to $\xi_x$, a ruled real hypersurface is such that $\cal D$ is integrable and its leaves are totally geodesic complex projective spaces $\mathbb{C}P^{n-1}$ in $\mathbb{C}P^n$. Among the works on ruled real hypersurfaces, we recall the one by 
Lohnherr and Reckziegel, \cite{LR}, who proved that any ruled real hypersurface can be parametrized from a certain curve on $\mathbb{C}P^n$, among other results. Some recent results concerning them are \cite{K3} and \cite{PB}. Summing up, both Hopf and ruled real hypersurfaces are two of the main families, because they appear in plenty of results. 

If we consider the indefinite complex projective space $\mathbb{C}P_p^n$ of index $1 \leq p \leq n-1$, see \cite{BR} for its Kaehlerian structure, the beginning of the study of its non-degenerate real hypersurfaces $M$ can be dated in 2015 by Anciaux and Panagiotidou, \cite{AP}. They obtained an almost contact metric structure on $M$ and proved a few basic results, like the non-existence of either totally umbilical real hypersurfaces or real hypersurfaces with parallel shape operator in $\mathbb{C}P_p^n$. They also posed a couple of problems. 

In \cite{KO}, Kimura and the second author obtained several families of Hopf real hypersurfaces in $\mathbb{C}P_p^n$, which were named {\it of type} $A_+$, $A_-$, $B_0$, $B_+$, $B_-$ and $C$ because they are somehow similar to those in Takagi's and Montiel's list \cite{M, T1}. They also showed an example of a degenerate (light-like) Hopf real hypersurface. This meant a positive answer to a question asked by Anciaux and Panagiotidou. All cases $A_+$, $A_-$, $B_0$, $B_+$ and $B_-$, were constructed as tubes of certain radius $r$ over some holomorphic totally geodesic submanifolds of complex dimension $m$. Also, the subindex $\pm$ indicates the causal character of the unit normal, except $B_0$, which is time-like Example $C$ is similar to the horosphere in the complex hyperbolic space, so they are also called horospheres. Next, in the same paper, they also proved that an $\eta$-umbilical ($AX=\lambda X+\varepsilon\mu\eta(X)\xi$, for any $X$ tangent to $M$, with $\xi$ normal and of unit length $\varepsilon=\pm 1$) non-degenerate real hypersurface in $\mathbb{C}P_p^n$ must be either of type $A_+$ or $A_-$ or $C$. They also classified real hypersurfaces whose Weingarten endomorphism is diagonalizable and satisfy that $\xi$ is a Killing vector field, answering another question posed in \cite{AP}.

Since Hopf real hypersurfaces in  $\cpn$ have been already introduced, the purpose of the present paper is to begin the study of \textit{ruled} real hypersurfaces in $\mathbb{C}P_p^n$, 
in the following way: 
\begin{enumerate}
\item Section \ref{axioms-of-planes} is devoted to obtaining three kinds of axioms of planes in $\mathbb{C}P_p^n$ that yield the existence of some families of totally geodesic submanifolds 
passing through a given point $x\in \mathbb{C}P_p^n$ and tangent to some vector subspaces of $T_x\mathbb{C}P_p^n$. We also introduce Frenet curves in $\mathbb{C}P_p^n$. 

\item In Section \ref{ruled}, we define ruled real hypersurfaces in $\mathbb{C}P_p^n$ as non-degenerate real hypersurfaces whose maximal holomorphic distributions are integrable with totally geodesic leaves in the ambient space, obtaining a characterization of such real hypersurfaces in terms of the shape operator. We also make a description of the shape operator of such hypersurfaces. In Subsection \ref{parametrization}  we introduce ruled hypersurface \textit{parametrizations} in $\mathbb{C}P_p^n$ induced by unit speed curves. We show that for a ruled real hypersurface $M$ in $\mathbb{C}P_p^n$ and each point $q \in M$, there exists a unique parametrization of $M$ around $q$, centered by an integral curve $\alpha$ of the Reeb vector field $\xi$. All this makes sense due to the axiom of planes for maximal holomorphic tangent planes, as in Section \ref{axioms-of-planes}. 

\item In Section \ref{minimal}, Theorem \ref{TH1} states the classification of minimal ruled real hypersurfaces, obtaining three possibilities for the integral curves of $\xi$. Analyzing such three cases and using the axioms of planes in Section \ref{axioms-of-planes}, in Section \ref{examples} we introduce several examples of ruled real hypersurfaces in $\mathbb{C}P_p^n$, which depend on the causal character of the curve. These curves can be either a geodesic, a totally real circle (as in \cite{A-O'M}) contained in a totally real, totally geodesic surface, or a non-Frenet curve contained in a totally real, totally geodesic 3-submanifold. 
\end{enumerate}

We also would like to remark the technique provided by M.~Kimura et al. in \cite{K3}. They established a one-to-one correspondence between ruled real hypersurfaces in the Riemannian complex hyperbolic space $\mathbb{C}H^n$, and some curves in the indefinite complex projective space of index $2$. Their techniques involve the family of geodesics in $\mathbb{C}H^n$. This point of view is quite interesting, and certainly will shed new light to the study of ruled real hypersurfaces. For us, since this paper is long enough, we will leave that for a future work. 

Finally, the authors would like to thank the referees, because their valuable comments improved this paper. 

\section{Set Up and Basic Results} \label{setup}

\subsection{The indefinite complex projective space $\mathbb{C}P^n_p$}
Let $\CC^{n+1}_p$ be the Euclidean complex space endowed with the following Hermitian product and pseudo-Riemannian metric of index $2p$, 
$z=(z_	1,\ldots,z_{n+1})$, $w=(w_1,\ldots,w_{n+1})\in\CC^{n+1}$, 
\begin{equation}\label{innerp} 
g_{\CC}(z,w)=-\sum_{j=1}^p z_j\bar{w}_j+\sum_{j=p+1}^{n+1}z_j\bar{w}_j, \quad 
g=\mathrm{Re}(g_{\CC}),
\end{equation}
where $\bar{w}$ is the complex conjugate of $w\in\CC^{n+1}$.  The natural complex structure will be denoted by $J$. As usual, we define the set $\mathbb{S}^1=\{a\in\CC : a\bar{a}=1\} = \{\mathrm{e}^{i\theta} : \theta\in\mathbb{R}\}$. We consider the hyperquadric 
\[ \mathbb{S}^{2n+1}_{2p}=\{z\in\CC ^{n+1}_p : g(z,z)=1\},
\]
which is a semi-Riemannian manifold of index $2p$. We define the action and its corresponding quotient
\[\mathbb{S}^1\times\mathbb{S}^{2n+1}_{2p}\rightarrow \mathbb{S}^{2n+1}_{2p}, \ (a,(z_1,\ldots,z_{n+1}))\mapsto (az_1,\ldots,az_{n+1}), 
\]
\[ \pi:\mathbb{S}^{2n+1}_{2p}\rightarrow \cpn=\mathbb{S}^{2n+1}_{2p}/\sim.
\]
Let $g$ be the metric on $\cpn$ such that $\pi$ becomes a semi-Riemannian submersion. The manifold $\cpn$ is called the \textit{Indefinite Complex Projective Space}. See \cite{BR} for details. We need $1\leq p\leq n-1$ to avoid $\mathbb{C}P^n$ with either a Riemannian or a negative definite metric. Let $\tilde\nabla$ be its Levi-Civita connection. Then, $\cpn$ admits a complex structure $J$ induced by $\pi$, with Riemannian tensor
\begin{align}\nonumber 
\bar{R}(X,Y)Z =& g(Y,Z)X-g(X,Z)Y \\ & +g(JY,Z)JX-g(JX,Z)JY+2g(X,JY)JZ, \label{RR}
\end{align}
for any $X,Y,Z\in TM$. Thus, $\cpn$ has constant holomorphic sectional curvature $4$. By an abuse of notation, we will denote by $J$ all the complex structures. It is important to bear in mind that for $q\in \snp$, 
\[ T_q\snp =\{ X\in \cn : g(X,q)=0\}.\] 

\noindent We give the following 
definitions.
\begin{definition}
In any complex manifold $(\bar{M},\bar{g},J)$, a tangent plane at $z\in \bar{M}$, $\Pi\subset T_z\bar{M}$, $\dim \Pi\geq 2$, is called \textit{holomorphic} if $J\Pi=\Pi$.
\end{definition}

\begin{definition}
In any complex manifold $(\bar{M},\bar{g},J)$, a tangent plane at $z\in \bar{M}$, $\Pi\subset T_z\bar{M}$, $\dim \Pi\geq 2$, is called  \textit{totally real} if $J\Pi\perp \Pi$.
\end{definition}

\noindent Notice that, according to the above definitions, a submanifold $S\subset \bar{M}$ is \textit{holomorphic} (resp. \textit{totally real}) when for each $p\in S$, $T_pS$ is holomorphic (resp. totally real). 

\begin{definition} Given a point $x\in \cpn$, we call a tangent vector $v\subset T_x\cpn$ to be \emph{space-like} if $g(v,v)>0$ or $v=0$, \emph{time-like} if $g(v,v)<0$, and \emph{light-like} if $\big(g(v,v)=0$ and $v\neq 0\big)$. Also, a tangent plane $\Pi\subset T_p\cpn$ is called \emph{space-like} if the induced metric on $\Pi$ is positive definite, \emph{time-like} if the induced metric is Lorentzian, and \emph{degenerate} if the induced metric is degenerate. Accordingly, a submanifold will be called \emph{non-degenerate} when the induced metric is not degenerate.
\end{definition}

\noindent Given $x\in \cpn$, let  $q\in \mathbb{S}^{n+1}_{2p}$ such that $\pi(q)=x.$ We have 
\begin{equation}
\pi:\sn\to\cpn \quad \text{ and }\quad d\pi_{q}:T_q\sn\to T_x\cpn.
\end{equation}
In general we will refer to $d\pi_q$ as to the restriction 
\begin{equation}
d\pi_q:(\mathrm{Ker}(d\pi_q))^{\perp}\to T_x\cpn,
\end{equation}
where $\mathrm{Ker}(d\pi_q)=\mathrm{span}\{J\vec q\}$, since this restriction is an isometry. Therefore, the map $\pi$ is a semi-Riemannian submersion, i.e., $\pi$ has maximal rank (each derivative map $d\pi$ is surjective) and $d\pi$ preserves lengths of horizontal vectors.

Let $D$ and $\hat{\nabla}$ be the Levi-Civita connections of $\mathbb{C}^{n+1}_p$ and $\snp$, respectively. Since the position vector $\chi:\snp\to \mathbb{C}^{n+1}_p$ behaves as a space-like, unit, normal vector field, the Gauss formula becomes
\begin{equation} D_XY = \hat{\nabla}_XY-\langle X,Y\rangle \chi, \label{GaussEq} \end{equation}
for any $X,Y$ tangent to $\snp$.

We will define next a \textit{Frenet} curve of order $r\geq 1$ in $\cpn$. When $r>1$, we are assuming that the Frenet vectors involved in the Frenet system of equations are never light-like. Given a curve $\alpha:I\to\cpn$, we denote by $D/ds$ the covariant derivative along it. 
\begin{definition}[\textbf{Frenet curves of order $r\geq 1$ in $\cpn$}]\label{lemmaFrenet} 
Given a unit curve $\alpha:I\to \cpn$, calling $\alpha'=E_1$, we have: If $\alpha$ is a geodesic, then $r=1$, its Frenet vector is $\{E_1\}$ and its Frenet curvature is $\kappa_1=0$. If $\alpha$ is not a geodesic, then 
\begin{align*}\label{Frenetsystem}
\frac{DE_1}{ds}&=\varepsilon_2\kappa_1 E_2,\\
\frac{DE_2}{ds}&=-\varepsilon_1\kappa_1 E_1+\varepsilon_3\kappa_2 E_3,\\
\frac{DE_3}{ds}&=-\varepsilon_2\kappa_2 E_2+\varepsilon_4\kappa_3 E_4,\\
&\vdots\\
\frac{DE_r}{ds}&=-\varepsilon_{r-1}\kappa_{r-1} E_{r-1},
\end{align*}
where $\{E_1,\ldots,E_r\}$ are orthonormal vectors along $\alpha$ with lengths $g(E_i,E_i)=\varepsilon_i=\pm 1$ and $\kappa_1,\ldots,\kappa_{r-1}$ their Frenet curvatures. Together they describe the \emph{Frenet system} of the Frenet curve $\alpha$ of order $r$. If $r=n$, then the last vector $E_r=E_n$ is chosen  such that 
$E_1,\ldots,E_n$ form a positive basis. All the Frenet curvatures $\kappa_k>0$, except if $r=n$, when $\kappa_{n-1}$ might be positive or negative. 
\end{definition}

\subsection{Real Hypersurfaces in $\mathbb{C}P^n_p$}\label{realhypersurfaces}

For details, see \cite{KO}. Let $M$ be a connected, non-degenerate, immersed real hypersurface in $\cpn$. If $N$ is a local unit normal vector field such that $\ep=g(N,N)=\pm 1$, we define the \textit{structure} vector field on $M$ as $\xi=-JN$. Clearly, $g(\xi,\xi)=\ep$. Given $X\in TM$, the vector $JX$ might not be tangent to $M$. Then, we decompose it in its tangent and normal parts, namely
\[ JX = \phi X+\ep\eta(X)N,
\]
where $\phi X$ is the tangential part, and $\eta$ is the 1-form on $M$ such that 
\[\eta(X)=g(JX,N)=g(X,\xi).\]
The set $(g,\xi,\phi,\eta)$ is called an almost contact metric structure, whose properties are
\begin{gather} \eta(\phi X)=0, \ \phi^2 X = -X+\ep\eta(X)\xi, \nonumber \\
g(\phi X,\phi Y) =g(JX-\ep\eta(X)N,JY-\ep\eta(Y)N) =g(X,Y)-\ep\eta(X)\eta(Y), \label{structureEq} 
\\ 
g(\phi X,Y)+g(X,\phi Y)=0, \ 
\phi\xi=0, \  \eta(\xi)=\ep, \nonumber \\
(\nabla_X\phi )Y=\varepsilon g(Y, \xi)AX-\varepsilon g(AX,Y)\xi, \nonumber
\end{gather}
for any $X,Y\in TM$. Next, if $\tilde{\nabla}$ and $\nabla$ are the Levi-Civita connections of $\mathbb{C}P^n_p$ and $M$, respectively, we have the Gauss and Weingarten formulae:
\begin{equation}\label{GaussWeingarten}
\tilde\nabla_XY=\nabla_XY+\ep g(AX,Y)N, \quad
\tilde\nabla_X N  = -AX,\end{equation}
for any $X,Y\in TM$, where $A$ is the shape operator associated with   $N$. Also, 
\begin{equation}\label{nablaxxi}
 \nabla_X\xi=\phi AX, 
\end{equation}
for any $X\in TM$. 
The Codazzi equation is
\begin{align}\label{cdz}
(\nabla_XA)Y-(\nabla_YA)X = \eta(X)\phi Y-\eta(Y)\phi X+2g(X,\phi Y)\xi,
\end{align}
for any $X,Y\in TM$. Let $R$ be the curvature operator of $M$. Then, by using \eqref{RR}, \eqref{GaussWeingarten}  and the structure Gauss equation \cite[pag. 100]{ONeill}, we obtain 
\begin{align*}
R(X,Y)Z=&g(Y,Z)X-g(X,Z)Y+g(\phi Y,Z)\phi X-g(\phi X,Z)\phi Y\\
& -2g(\phi X,Y)\phi Z +\ep g(AY,Z)AX -\ep g(AX,Z)AY,
\end{align*}
for any $X,Y,Z\in TM$.
\section{Some Axioms of Planes} \label{axioms-of-planes}

The following 3 lemmatta state that $\cpn$ satisfies certain types of \textit{axioms of planes}. Essentially, we show the existence of some families of totally geodesic submanifolds, passing through a given point and tangent to some tangent planes. The ideas underneath come from the fact that totally geodesic submanifolds of some model spaces embedded in flat $\RR^n_m$ can be computed by  intersecting these models with linear subspaces. For more details, see \cite{ONeill}.

\begin{lemma}\label{firstlemma}
Let $x\in \mathbb{C}P^n_p$ and consider a holomorphic, non-degenerate hyperplane $\Pi\subset T_x\mathbb{C}P^n_p$ of dimension $\dim_{\mathbb{R}}\Pi=2n-2$.  The following hold:
\begin{enumerate}
\item There exists a totally geodesic submanifold $L$ in $\cpn$ such that $x\in L$ and $T_xL=\Pi$. In addition, $L$ is isometric to $\mathbb{C}P^{n-1}_t$, where $t=p$ if $\Pi^{\perp}$ is definite positive  or $t=p-1$ if $\Pi^{\perp}$ is definite negative. 

\item Any two totally geodesic submanifolds $L_1$ and $L_2$ as above are linked by an isometry of $\cpn$. 
\end{enumerate}
\end{lemma}
\begin{proof} \textit{Step 1:} 
We start by pointing out the following totally geodesic embeddings:
\begin{align*} &\chi_{+}: \mathbb{S}_{2p}^{2n-1} \to \snp, \ (z_1,\ldots,z_n)\mapsto (z_1,\ldots,z_n,0), \\ 
& \chi_{-}:\mathbb{S}_{2p-2}^{2n-1} \to \snp, \ (z_1,\ldots,z_n)\mapsto (0,z_1,\ldots,z_n).
\end{align*}
It is quite clear that we obtain the following totally geodesic embeddings:
\[ \mathbb{C}P_{p}^{n-1}\to \cpn, \quad \mathbb{C}P_{p-1}^{n-1}\to \cpn.\]
If we denote $o=(0,\ldots,0,1)\in\mathbb{S}_{2p}^{2n-1}$, $Q_+=(0,\ldots,0,1),Q_{-}=(1,0,\ldots,0)\in\mathbb{C}_{p}^{n+1}$, then 
\begin{align*}
& T_o\mathbb{S}_{2p}^{2n-1} = \{X\in\mathbb{C}^n_p : \mathrm{Re}(X_n)=0\}, \
T_{(o,0)}\mathbb{S}_{2p}^{2n+1} = \{X\in\mathbb{C}^{n+1}_p : \mathrm{Re}(X_n)=0\}, \\
& (d\chi_{+})_o(X)=(X,0), \ \forall X\in T_o\mathbb{S}_{2p}^{2n-1},\\
& (d\chi_+)_o\big( T_o\mathbb{S}_{2p}^{2n-1}\big)  = \{Z\in \mathbb{C}^{n+1}_p :\mathrm{Re}(Z_{n})=0, Z_{n+1}=0\} = 
\mathrm{Span}\{(o,0),Q_+,JQ_+\}^{\perp}.
\end{align*}
There are similar expressions for $\mathbb{S}_{2p-2}^{2n-1}$, using $Q_{-}$.

\textit{Step 2:} Take $x\in\cpn$, and a non-degenerate $\Pi\subset T_{x}\cpn$, such that $J\Pi=\Pi$, $\dim_{\RR}\Pi=2n-2$. Consider ${q}\in\snp$   such that $\pi({q})=x$. There exists a unit $\eta\in T_{x}\cpn$ such that $\Pi^{\perp}=\mathrm{Span}\{\eta,J\eta\}$. Let $\hat{\eta}$ be its horizontal lift via $\pi$. Construct an  $\RR$-orthonormal basis $B=(v_{1},Jv_{1},\ldots,v_{n+1},Jv_{n+1})$ in $\CC^{n+1}_p$ in the following way. If $\hat\eta$ is space-like, then $v_{n}={q}$, $v_{n+1}=\hat\eta$. If $\hat\eta$ is time-like, then $v_{n}={q}$, $v_1=\hat\eta$. Define the matrix $M$ such that its $k^{th}$ column is $v_k$. 
If $\det(M)=-1$, we change a suitable $v_k$ by $-v_k$. 
Since $B$ is a real orthonormal basis, it is easy to check that $M\in SU(p,\bar{p})=\{M\in \mathcal{M}_{n+1}(\mathbb{C}) \, | \, \bar{M}^tI_{p,\bar p}M=I_{p,\bar p}, \, \det M=1 \}$, for $\bar p=n+1-p$, {where $I_{p,\bar{p}}$ is the matrix $I_{p,\bar{p}}=\mathrm{Diagonal}(-1,\stackrel{(p}{\ldots},-1,1,\stackrel{(\bar{p}}{\ldots},1)$.} The associated linear map $f_{M}:\cn\to \cn$ can be restricted to $\snp$ as an isometry, satisfying $f_M(o,0)={q}$, $(df_M)_{(o,0)}Q_{\pm}=\hat\eta$. In addition, it can be projected to $\cpn$, obtaining a holomorphic isometry $f:\cpn\to\cpn$ such that $f(\pi(o,0))=x$ and $f_*(\pi_*(Q_{\pm}))=\eta$. 

\textit{Step 3:} The desired totally geodesic submanifold is $L_x=f(\pi(\mathbb{S}_t^{2n-1}))$, for $t=p$ if $\eta$ is space-like or $t=p-1$ if $\eta$ is time-like This totally geodesic submanifold is unique at the point $x$, up to the chosen hyperplane. By composing two of the above isometries, any two totally geodesic submanifolds as above will be linked by a holomorphic isometry. \end{proof} 

We provide a list of totally real, totally geodesic surfaces in $\cpn$. In all cases, we denote $\hat{\Sigma}$ the surface embedded in $\mathbb{S}^{2n+1}_{2p}$, and its projection $\Sigma=\pi(\hat{\Sigma})$, as in the following commutative diagram
\begin{center}
\begin{tikzcd}
\hat{\Sigma}  \arrow[r, "t.g.\ \hat{\varphi}", description]   \arrow[d, "\pi", description ]  & 
\mathbb{S}^{2n+1}_{2p}  \arrow[d, "\pi", description]   \\
\Sigma  \arrow[r, "t.g.\ \varphi", description] & \mathbb{C}P^n_p  \\
\end{tikzcd}
\end{center}
\vspace{-1em} We will follow \cite[p.110]{ONeill}. {We also recall that a map $f:(\bar{M},\bar{g})\to(\bar{M},\bar{g})$ is an \textit{anti-isometry} if for each $p\in\bar{M}$ and each $u,v\in T_p\bar{M}$, it holds $g_p(u,v)=-g_{f(p)}(f_*u,f_*v)$. As a consequence, the sectional curvatures of $2$-planes linked by $f$ have different signs.} 

As above, all embeddings in $\cpn$ are totally real and totally geodesic.

\begin{enumerate}
\item[$\mathbb{R}P^2$)] For $n\geq 3$, $1\leq p\leq n-2$, the round sphere $\hat{S}_1=\{(x,y,z)\in\RR^3 : x^2+y^2+z^2=1 \}$, and 
$\hat{\varphi}_1:\hat{S}_1\to \mathbb{S}^{2n+1}_{2p}$, $(x,y,z)\mapsto (0,\ldots,0,x,y,z)$. By the classical quotient $\mathbb{R}P^2=\hat{S}_1/\{\pm I_3\}$, $I_3$ the identity matrix of order $3$, we get an embedding of the Riemannian  real projective plane $\varphi_1:\mathbb{R}P^2\to \cpn$. 

\item[$\mathbb{H}_2^2$)] For $n\geq 3$, $2\leq p\leq n-1$, let $\hat{S}_2	=\{(x,y,z)\in\RR^3_2 : -x^2-y^2+z^2=1 \}$, and $\hat{\varphi}_2:\hat{S}_2\to \mathbb{S}^{2n+1}_{2p}$, $(x,y,z)\mapsto (x,y,0,\ldots,0,z)$. 
Take the classical quotient $\mathbb{H}_2^2:=\hat{S}_2/\{\pm I_3\}$. Then, we get $\varphi_2:\mathbb{H}_2^2\to\cpn$. By Lemma 24 of \cite[p.110]{ONeill}, $\mathbb{H}_2^2$ is anti-isometric to the standard real hyperbolic plane.

\item[$\mathbf{S}^2_1$)] For $n\geq 2$, $1\leq p\leq n-1$, take $\hat{S}=\{(x,y,z)\in\RR^3_1 : -x^2+y^2+z^2=1\}$. Its universal covering is the \textit{de Sitter} 2-plane. The embedding $\hat{\varphi}_3:\hat{S} \to \mathbb{S}^{2n+1}_{2p}$, $(x,y,z)\mapsto (x,0,\ldots,0,y,z)$, induces $\varphi_3:\mathbf{S}^2_1=\hat{S}/\{\pm I_3\}\to \cpn$. 
\end{enumerate}
\begin{lemma} \label{totgeosurf} 
Given $x\in \cpn$, consider a non-degenerate, totally real, 2-plane $\Pi\subset T_x\cpn$. 
\begin{enumerate}
\item There exists a non-degenerate, totally real, totally geodesic surface $\Sigma\subset\cpn$ such that $x\in \Sigma$ and $T_x\Sigma = \Pi$.
\item Any two surfaces $\Sigma_1$ and $\Sigma_2$ as in item 1 are linked by an isometry of $\cpn$ if, and only if, they  have one causal character among the previous cases $\mathbb{R}P^2$, $\mathbb{H}_2^2$ and $\mathbf{S}^2_1$.
\end{enumerate}
\end{lemma}
\begin{proof} Since $\cpn$ has constant holomorphic sectional curvature $+4$, totally real, totally geodesic submanifolds have constant sectional curvature $+1$. Then, we look for the model spaces of surfaces with constant Gaussian curvature $K=+1$, with index $0$, $1$ and $2$. These are described in the previous list. Finally, the proof can be finished by a similar technique to  Lemma \ref{firstlemma}. 
\end{proof}

Next, we show the list of totally real, totally geodesic, 3-dimensional submanifolds in $\cpn$ with index 1 or 2. We will not repeat that all the embeddings will be totally real and totally geodesic. The ideas are very similar to the previous list of surfaces. 
\begin{enumerate}[($\mathbf{B}^3_1$)]
\item For $n\geq 3$, $1\leq p\leq n-2$, take the de Sitter space $\hat{\psi}_1:dS^3=\{p=(x_1,x_2,x_3,x_4)\in \mathbb{R}^4_1 : -x_1^2+x_2^2+x_3^2+x_4^2=1\}\to \mathbb{S}_{2p}^{2n+1}$, $p\mapsto (x_1,0,\ldots,0,x_2,x_3,x_4)$. This induces $\mathbf{B}^3_1=dS^3/\{\pm I_4\}\to \cpn$, where $I_4$ is the identity matrix of order $4$. The index of $\mathbf{B}_1^3$ is $1$. 

\item For $n\geq 3$, $2\leq p\leq n-1$, take the hyperquadric $Q^3_2=\{p=(x_1,x_2,x_3,x_4)\in\RR^4_2 : -x_1^2-x_2^2+x_3^2+x_4^2=1\}\to \mathbb{S}^{2n+1}_{2p}$ so we obtain  
$\mathbf{B}^3_2=Q^3_2/\{\pm I_4\}\to \cpn$. The index is 2. By Lemma 24 \cite[p.110]{ONeill}, 
 $\mathbf{B}^3_2$ is locally anti-isometric to the \textit{anti-de Sitter} 3-space. 
\end{enumerate}
\begin{lemma}\label{3submanifolds} 
Given $x\in \cpn$, consider a non-degenerate, totally real 3-plane $\Pi\subset T_x\cpn$ of index 1 or 2. 
\begin{enumerate}
\item There exists a non-degenerate, totally real, totally geodesic, 3-submanifold $B\subset\cpn$ such that $x\in B$ and $T_xB = \Pi$.
\item Any two 3-submanifold $B_1$ and $B_2$ as in item 1 are linked by an isometry of $\cpn$ if, and only if, they  have the same causal character, among the previous cases $\mathbf{B}^3_1$ and $\mathbf{B}^3_2$.
\end{enumerate}
\end{lemma}
\begin{proof} 
Since the index has to be 1 or 2, we consider the family of simply connected, 3-manifolds of constant sectional curvature of index 1 or 2. We can use \cite[p.110]{ONeill} to construct them. When the index is 1, then it is (locally) the de Sitter 3-space. When the index is 2, then it is (locally) $Q_2^3$. The previous list provides some model immersions. 

We finish the proof again as in Lemma \ref{firstlemma}. 	
\end{proof}
\color{black}

\section{Ruled real hypersurfaces}\label{ruled}
In this section we will first give the definition of \textit{ruled real hypersurfaces} and present some basic characterizations. We will determine the form of the shape operator and then we will discuss the parametrization of ruled real hypersurfaces.

Let $M$ be a hypersurface in $\cpn$ and let $q\in M$. The maximal holomorphic distribution is defined as  $\mathcal{D}_q:=T_qM\cap JT_qM=\mathrm{span}\{\xi_q\}^{\perp}$.

\begin{definition}
	A non-degenerate real hypersurface $M$ in $\cpn$ is \textrm{ruled} if, and only if, the maximal holomorphic distribution $\mathcal D$ is integrable with totally geodesic leaves in $\cpn$.
\end{definition}

\begin{lemma}
Let $M$ be a non-degenerate  real hypersurface. Then $M$ is ruled if, and only if, $g(AX,Y)=0,$ $\forall  X,Y\in \mathcal{D}$.
\end{lemma}
\begin{proof} Assume that $M$ is ruled. Let $L$ be a totally geodesic leaf of $\mathcal{D}$ in $\cpn$. We call $\nabla^L$ its Levi-Civita connection. 
By Gauss equation, since $L\subset \cpn$, for any $X,Y\in TL$, we have $\tilde{\nabla}_XY=\nabla^L_XY$, which proves that $\tilde{\nabla}_XY\in TL$. Moreover, since $\mathcal{D}\subset TM\subset \cpn$ and  $ L\subset M\subset\cpn$, we have that $$\tilde{\nabla}_XY=\nabla_XY+\ep g(AX,Y)N.$$
As $N$ is not a light-like vector, we take the inner product of the above equation with $N$, and using that $N\perp\dd$, $TL\subset \dd$, we obtain $g(\tilde{\nabla}_XY, N)=g({\nabla}_XY,N)+ \ep ^2 g(AX,Y)$, that is to say, $0=g(AX,Y)$.

We assume now that $g(AX,Y)=0,$ $\forall  X,Y\in \mathcal{D}$. We easily obtain $0=g(AX,Y)=-g(\tilde{\nabla}_XN,Y)=g(N,\tilde{\nabla}_XY)$, for any $X,Y\in \mathcal{D}$, which implies that $\tilde{\nabla}_XY\in TM$.
We know that $JY=\phi Y\in \dd$, for any $Y\in \dd$. It follows that 
$0=g(AX,\phi Y)=-g(\tilde{\nabla}_XN,JY)=g(N, \tilde{\nabla}_XJY)=g(N, J\tilde{\nabla}_X Y)=g(\xi, \tilde{\nabla}_X Y).$ 
Therefore, 	$\tilde{\nabla}_X Y$ is tangent to $M$ and orthogonal to $\xi$, which implies that $\tilde{\nabla}_X Y\in \dd$. Similarly, we show that $\tilde{\nabla}_YX\in \dd$. Hence, for $X,Y\in \dd$, $[X,Y]=\tilde{\nabla}_X Y-\tilde{\nabla}_YX\in \dd$ and so $\dd$ is an integrable distribution. Next, let us see that the leaves of $\dd$ are totally geodesic. Take $L$ as one of them, that is, $L$ is a submanifold of $\cpn$ such that for any $p\in L$, we have $T_pL=\dd_p$. Let $\nabla^{L}$ and $\sigma$ be the Levi-Civita connection on $L$ and the second fundamental form of $L$ in $\cpn$. Therefore, for $X,Y\in TL$, the Gauss equation writes as 
\begin{align*}
\tilde{\nabla}_XY=\nabla^{L}_XY+\sigma(X,Y).
\end{align*}
In fact, for $X,Y\in \Gamma(TL)$, we have that $ \nabla^{L}_XY\in \Gamma(TL)$ and $\tilde{\nabla}_X Y|_{q}\in \dd_q=T_qL,$ for $q\in L$. 
It follows that  $\sigma(X,Y)=0$ and, hence, $L$ is totally geodesic. Finally, by Lemma \ref{firstlemma}, the leaves are isometric to some $\mathbb{C}P^{n-1}_t$, where $t=p$ if $N$ is space-like or $t=p-1$ if $N$ is time-like. 
\end{proof}

\subsection{The shape operator of ruled real hypersurfaces}\label{opereitor}

From now on, $M$ will denote a ruled real hypersurface in $\cpn$, unless otherwise stated. 
Suppose there exists an open subset $\Omega$ in $M$ on which the shape operator of $M$ vanishes: $A|_{\Omega}\equiv 0$. It implies that $\Omega$ is totally geodesic, but this leads to a contradiction as follows. For $X,Y\in TM$, the Codazzi equation \eqref{cdz} implies that $0=\eta(Y)\phi X-\eta(X) \phi Y-2 g(X,\phi Y)\xi$ and we obtain immediately that $g(X,\phi Y)=0$, for all $X,Y\in TM$. We see that for a unit vector $X\in \mathcal{D}$ and $Y=\phi X$, we obtain a contradiction. 
Then, on a (dense) open subset $\Omega$, there exists a non-zero $U\in \dd$, and a function $\mu\in C^{\infty}(\Omega)$ such that
\begin{equation} A\xi=\ep \mu \xi+ U. 
\end{equation}
We discuss the shape operator depending on the causal character of the vector field $U$. \\

$\bullet$ Suppose first that $U$ is time-like or space-like. There exists a unit $W\in \mathcal{D}$ and a function $\lambda\in C^{\infty}(\Omega)$ such that $U=\ep_W\lambda W$, $\ep_W=g(W,W)=\pm 1$. Then, $A\xi =\ep \mu \xi +\ep_W\lambda W$. Since $g(AW,X)=0$ for any $X\in\dd$, $AW=\ep g(AW,\xi)\xi=\ep g(W,A\xi)\xi = 
\ep g(W,\ep \mu \xi +\ep_W\lambda W)\xi= \ep\lambda \xi$. We may write 
\begin{equation}\label{shapeOpts}
A\equiv\begin{pmatrix} \varepsilon \mu & \varepsilon \lambda  & 0 & \ldots & 0 \\
\ep_W\lambda &  &  & &  \\
0 &  &\bigzero & &  \\
\vdots & & &&\\ 
0 &  & & &  \\
\end{pmatrix}, \quad A\xi=\ep \mu \xi + \ep_W\lambda W, \ AW = \ep\lambda \xi. 
\end{equation}
Remark that $g(\xi,A\xi)=\mu$ and $g(AW,\xi)=\lambda$. Suppose first that $\mathrm{rank}(A)=1$ on an open set $\tilde{\Omega}\subset M$.
It follows that $\lambda=0$ on $\tilde{\Omega}$ and therefore $A\xi=\ep\mu\xi$, $AX=0$, for all $X\perp \xi.$ 
However, when $AX$ is of the form $AX=a X+b\eta(x)\xi$, $a,b\in C^{\infty}(M)$, we know that $M$ is in the list of  Theorem 4.3, \cite{KO}, which requires that $a$ is always non-zero. This does not hold when $\mathrm{rank}(A)=1$.
We conclude that there exists a dense open subset $\Omega=\{q\in M : \lambda(q)\neq 0\}$ in $M$ such that $\mathrm{rank}(A|_{\Omega})=2$ and $A$ is given as in \eqref{shapeOpts}. We will work on this $\Omega$. 

Next, given $q\in \Omega$, consider an integral curve of $\xi$, namely, $\alpha:(-\delta,\delta) \to \Omega$, with $\alpha(0)=q$, $\alpha'(s)=\xi_{\alpha(s)}$, $\forall s\in (-\delta,\delta)$. Note that $\tilde \nabla _{\alpha'(s)}\alpha'(s)\neq 0$ because along the curve $\alpha$, 
\begin{equation*}
\tilde \nabla _{\alpha'(s)}\alpha'(s) = (\tilde \nabla_{\xi} \xi)_{\alpha(s)} =(\nabla_{\xi}\xi+\varepsilon g(A\xi,\xi)N)_{\alpha(s)}=(\phi A\xi+\ep \mu N)_{\alpha(s)}= (\ep_W\lambda \phi W+\ep \mu N)_{\alpha(s)}. 
\end{equation*}
The previous vector does not vanish since $\mathrm{rank}(A)\neq 0$ assures that $\lambda (\alpha(s))\neq 0$.

$\bullet$ Suppose now that $U$ is light-like and $A\xi=\ep \mu \xi+ U$. There exist orthonormal vectors $E_1$ and $E_2$ in $\dd$, and a function $a\in C^{\infty}(\Omega)$ such that  $g(E_1,E_1)=1$, $g(E_2,E_2)=-1$, $g(E_1,E_2)=0$, and $U=aE_1-aE_2$. Since $g(\phi E_1,\phi E_1)=g(JE_1,JE_1)=+1$, then $n\geq 3$. As before, $AE_i=\ep \rho_i \xi$, for some function $\rho_i$, $i=1,2$. But then, $\rho_i=g(AE_i,\xi)=g(E_i,A\xi)=g(E_i,\ep\mu\xi+aE_1 -a E_2)=a$.  Therefore we may find an orthonormal  basis $B=\{\xi,E_1,E_2,\ldots, E_{2n-2}\}$, with $g(E_i,E_i)=\pm 1$, for which the shape operator is of the form
 \begin{equation}\label{shapeOps}
A\equiv\begin{pmatrix}
\begin{matrix} \ep \mu \end{matrix} & \begin{matrix} \ep a& \ep a& 0 & \ldots & 0\end{matrix} \\  
\begin{matrix} a\\ -a\\  0\\ \vdots\\ 0 \end{matrix} & \bigzero 
\end{pmatrix}.
\end{equation}
 As a consequence, $AU=aAE_1-aAE_2=a(\varepsilon a\xi-\varepsilon a\xi)=0$ and 
$A\phi U=aA(\phi E_1 -\phi E_2)=\varepsilon a g(\phi E_1-\phi E_2,A\xi)\xi=\varepsilon a g(\phi E_1-\phi E_2,\varepsilon\mu\xi +aE_1-aE_2)\xi = \varepsilon a [ -a g(\phi E_1,E_2) -a g(\phi E_2,E_1)]=0.$

\subsection{The parametrization of ruled real hypersurfaces}\label{parametrization} 
Let $I$ be a real interval such that $0\in I$ and let $\alpha:I\to \cpn$ be a regular unit speed curve in $\cpn$, either space-like or time-like. Given $s\in I$, let $\Pi_s$ be the complex holomorphic non-degenerate hyperplane of maximal dimension in $\cpn$ tangent at $\alpha(s)$, such that $\Pi_s^{\perp}=\mathrm{span}\{\alpha '(s),J\alpha '(s)\}$. By Lemma \ref{firstlemma}, there exists a totally geodesic submanifold $L_s$  in $\cpn$  such that $T_{\alpha(s)}L_s=\Pi_s$.  We know that $L_s$ is isometric to $\mathbb{C}P^{n-1}_t$, with $t=p$, if $\Pi^{\perp}$ is positive definite or $t=p-1$, if $\Pi^{\perp}$ is negative definite. In this way, we obtain a family of totally geodesic submanifolds $\{L_s\}_{s\in I}$, with the property that $T_{\alpha(s)}L_s=(\mathbb C \alpha'(s))^{\perp}$. Intuitively, the set $\cup_{s\in I} L_s$ is a ruled real hypersurface. At first sight, there are two major families of ruled real hypersurfaces in $\cpn$, according to $\alpha'$ being time-like or space-like.
\begin{remark}\normalfont 
The examples of ruled real hypersurfaces constructed in the way described above satisfy that the maximal holomorphic distribution $\mathcal{D}$ is integrable, where  $\mathcal{D}=(\mathrm{span}\{\alpha', J\alpha'\})^{\perp}$.
\end{remark}
From now on, let $L_t^{n-1}$ be a connected open subset of $\mathbb{C}P_t^{n-1}$.
\begin{definition}\label{def1}
A smooth map $f:M:=I\times L_{t}^{n-1}\to \cpn$ is called {\em a ruled hypersurface parametrization} (shortly RHS-parametrization) induced by the curve $\alpha$ in the indefinite complex projective space $\cpn$, iff $f$ satisfies the following conditions:
	\begin{enumerate}
		\item The curve $\alpha $ is unit speed  $g(\alpha'(s), \alpha'(s))=\pm 1:= \varepsilon$. 
		\item There exists a point $x_0\in L_t^{n-1}$ such that $\alpha:I\to\cpn$, $\alpha(s)=f(s,x_0)$ for any $s\in I$; 
		\item For every $s\in I$, the map $f_s:L^{n-1}_t\to \cpn$, $q\mapsto f(s,q)$ is a totally geodesic and holomorphic immersion with $(f_{s})_*T_{x_0}L^{n-1}_t=(\mathrm{span}\{\alpha'(s), J\alpha'(s)\})^{\perp}$, where $t=p$ if $\varepsilon= 1$, or $t=p-1$ if $\varepsilon=-1$.
		\item For every $v\in T_{q_0}L^{n-1}_t$, the vector field $Z_v\in \mathfrak{X}_{\alpha}(\cpn)$, $s\mapsto {(Z_v)_s:=} {(df_s)}_{q_0}(v)$, along $\alpha$, is parallel in the bundle $(\mathrm{span}\{\alpha', J\alpha'\})^{\perp}$, i.~e. $Z_v(s)\in \mathrm{span}\{ \alpha'(s), J\alpha'(s) \}^{\perp}$ and $\tilde {\nabla}_{\alpha '(s)}Z_v\in \mathrm{span}\{\alpha'(s), J\alpha'(s)\}$, where $\tilde \nabla$ is Levi Civita connection in $T\cpn$. -- In other words, $f$ puts $L_t^{n-1}$ along $\alpha$ without rotations.
	\end{enumerate}
\end{definition}
\begin{proposition}\label{construction}
Given a unit speed curve $\alpha:I\to \cpn$, take $s_0\in I$. Then, there exists one, and only one, RHS-parametrization $f:I\times L^{n-1}_t\to \cpn$ such that: 
\begin{enumerate}
\item there exists $x_0\in \mathbb{C}P^{n-1}_t$ with $\alpha(s_0)=f_{s_0}(x_0)\in f_{s_0}(\mathbb{C}P^{n-1}_t)$, and 
$(f_{s_0})_*T_{x_0}\mathbb{C}P^{n-1}_t=\mathrm{span}\{\alpha'(s_0),J\alpha'(s_0)\}^{\perp}$;
\item $f(s,x)=\alpha(s)$ for any $s\in I$.
\end{enumerate}
\end{proposition}
\begin{remark}\normalfont 
The curve $\alpha$ might not be a Frenet curve. We are just asking $\alpha'(s)\neq 0$ and not light-like for all $s$. The vector field $\tilde{\nabla}_{\alpha'}\alpha'$ along $\alpha$ does exist, and can be space-like, light-like, time-like, zero at some points, and even a mixture of cases.
\end{remark} 
\begin{proof}We are assuming that $\ep_1:=g(\alpha',\alpha')=\pm 1$ is a constant function. For simplicity, we call $W=\mathrm{span}\{\alpha'(s_0),J\alpha'(s_0)\}^{\perp}$. If $\ep_1=1$, we take $t=p$, and if $\ep_1=-1$, we take $t=p-1$. Next, given $v\in W$, we consider the following initial value problem: 
\begin{align}\label{diff}
\tilde \nabla _{\alpha'(s)} Z_v(s)&=-\varepsilon_1\langle Z_v(s), \tilde{\nabla}_{\alpha'(s)}\alpha'(s)  \rangle \alpha'(s) - \varepsilon_1\langle Z_v(s), J\tilde{\nabla}_{\alpha'(s)}\alpha'(s) \rangle J\alpha'(s),\\
Z_v(s_0)&=v. \nonumber 
\end{align}
Compute its solution $Z_v:I\to T\cpn$. By \eqref{diff},
$\tilde{\nabla}_{\alpha'(s)}Z_v(s) \in \mathrm{span}\{\alpha'(s), J\alpha'(s) \}$. 
The solution for $v=0$ is just $Z_0(s)=0$. Next, by taking inner product with $\alpha'$ and $J\alpha'$ in \eqref{diff}, it is  simple to check
\[ \langle \tilde{\nabla}_{\alpha'}Z_v,\alpha'\rangle = - \langle Z_v,\tilde{\nabla}_{\alpha'}\alpha'\rangle, \quad 
\langle \tilde{\nabla}_{\alpha'}Z_v,J\alpha'\rangle = - \langle Z_v,J\tilde{\nabla}_{\alpha'}\alpha'\rangle.
\]
With these, we obtain immediately that
\begin{align*}
\frac{d}{ds}\langle Z_v(s), \alpha'(s) \rangle= \langle \tilde \nabla_{\alpha'(s)} Z_v(s) , \alpha'(s)\rangle+ \langle Z_v(s), \tilde{\nabla}_{\alpha'(s)}\alpha'(s) \rangle=0. 
\end{align*}
And at $s=s_0$, $\langle Z_v(s_0),\alpha'(s_0)\rangle = \langle v,\alpha'(s_0)\rangle =0$. This shows that $Z_v\perp\alpha'$. Similarly, 
$Z_v(s)\in \mathrm{span}\{\alpha'(s),J\alpha'(s)\}^{\perp}$ for any $s\in I$. It is important here that we can extend  this construction to $Z:I\times W\to T\cpn$, $Z(s,v)=Z_v(s)$, obtaining a smooth map. 

By Lemma \ref{firstlemma}, there exists a (unique) totally geodesic, holomorphic embedding $\Phi:\mathbb{C}P^{n-1}_t\to\cpn$ such that for some $x_0\in\mathbb{C}P^{n-1}_t$, $\Phi(x_0)=\alpha(s_0)$ and $\Phi_*(T_{x_0}\mathbb{C}P^{n-1}_t)=W$. By restricting to a suitable open subset $L_t^{n-1}$ of $\mathbb{C}P^{n-1}_t$ with $x_0\in L_t^{n-1}$, we obtain an injective map $r:L_t^{n-1}\to T_{x_0}\mathbb{C}P^{n-1}_t$ such that $(\exp_{x_0})^{-1}=r$. We construct the injective map $\Psi:L^{n-1}_t\to W$, $\Psi:=(d\Phi)_{x_0}\circ r.$ By combining these maps, we define 
\[
f:I\times L^{n-1}_t\to \cpn,\quad 
f(s,x)=\exp_{\alpha(s)}(Z(s,\Psi(x))),\]
which clearly satisfies $f(s,x_0)=\exp_{\alpha(s)}(Z_0(s))=\alpha(s)$ for any $s\in I$, and for each $q\in L^{n-1}_t$, 
\begin{align*}  &f(s_0,q)=\exp_{\alpha(s_0)}(Z(s_0,\Psi(q)))=
\exp_{\alpha(s_0)}(\Psi(q)) \\ & =\exp_{\Phi(x_0)}\big((d\Phi)_{x_0}(r(q))\big)=\Phi(\exp_{x_0}(r(q)))=
\Phi(q). \end{align*}
For every $s\in I$,  the image of $Z(s,\cdot)$ is a holomorphic linear subspace of maximal dimension, thus 
$f(s,\cdot)$ is a totally geodesic submanifold, isometric to an  open subset of $\mathbb CP_t^{n-1}$.
\end{proof}
\begin{remark}\normalfont The uniqueness in these previous results are \textit{up to an isometry of $\mathbb{C}P^{n-1}_t$}. Indeed, it is well known that totally geodesic hyperplanes $\mathbb{C}P^{n-1}$ are invariant by some subgroups of isometries of $\mathbb{C}P^n$. Similarly, 
totally geodesic hyperplanes  $\mathbb{C}P^{n-1}_t$ are invariant by some subgroups of isometries of $\cpn$. 
\end{remark}

\begin{corollary} 
	Let $M$ be a ruled real hypersurface in $\cpn$. Then, for each point $q\in M$ there exists a unique parametrization of $M$ around $q$ as in Definition \ref{def1}.
\end{corollary}
\begin{proof}
We just need to consider the integral curve $\alpha:(-\delta,\delta)\to \cpn$ such that $\alpha(0)=q$ and $\alpha'=\xi_{\alpha}$. We recall Proposition \ref{construction} to construct the parametrization. Given a point and a holomorphic hyperplane, by Lemma \ref{firstlemma}, the totally geodesic submanifolds $\mathbb{C}P^{n-1}_t$ are unique, so the RHS-parametrization is a parametrization of $M$ around $q$. 
\end{proof}

\section{Minimal Ruled Real Hypersurfaces}\label{minimal}

We introduce the following definition, inspired by \cite{A-O'M}.
\begin{definition}
\label{totrealcircle}
A \emph{totally real circle} is a Frenet curve of order $2$, with constant curvature and such that the two Frenet vectors $F_1, F_2$ (see Definition \ref{lemmaFrenet}) span a totally real plane at each point of the curve. \end{definition}
\begin{definition} A real hypersurface $M$ in $\cpn$ will be called \emph{minimal} if its mean curvature (function) vanishes everywhere. 
\end{definition}
\begin{theorem}\label{TH1}
	A minimal ruled real hypersurface $M$ in $\cpn$ is generated by a unit curve $\alpha:\de\to M$ which satisfies one of the following conditions:
	\begin{itemize}
		\item[a)]  $\alpha$ is a space-like or  time-like geodesic;
		\item[b)]  $\alpha$ is a totally real circle of a totally real, totally geodesic $\mathbb{R}P^2$, $\mathbb{H}^2_{2}$ or $\mathbf{S}^2_1$ in $\cpn$; 
		\item[c)] $\alpha$ is a totally real curve, but not a Frenet curve, determined by the system 
		\begin{equation}\label{curve-c}
		F_1:=\alpha', \ F_2:=\tilde{\nabla}_{F_1}F_1,\ \tilde{\nabla}_{F_1}F_2=0,\  
		F_2\ {light-like}, \ g(F_1,F_2)=0, \  g(F_1,J F_2)=0,
		\end{equation} 
		contained in a totally real, totally geodesic submanifold of $\cpn$, namely $\mathbf{B}_1^3$ if $\alpha$ is space-like, or $\mathbf{B}^3_2$ if $\alpha$ is time-like. 
	\end{itemize}
\end{theorem}
\begin{proof} As in Section \ref{opereitor}, $A\xi=\ep \mu\xi+U$, for some $U\in \dd$. Since $g(AX,Y)=0$ for any $X,Y\in \dd$, it follows that for a minimal ruled real hypersurface we have
\begin{align}
A\xi=U,\quad AX=\ep g(AX,\xi)\xi=\ep g(X,U)\xi, \quad \phi AX=0,
\end{align} 
for any $X\in \dd$. We also need 
\begin{align*}
\tilde{\nabla}_\xi\xi=\nabla_\xi\xi+\varepsilon g(A\xi,\xi)N=\phi A\xi=\phi U. 
\end{align*}
We will discuss two cases, according to the causal character of  the vector field $U$.\\
\textbf{Case 1}: Assume $U$ is space-like or time-like. We have
\begin{align*}
\tilde{\nabla}_{\xi}\tilde{\nabla}_{\xi}\xi=& \tilde{\nabla}_\xi\phi U=\nabla_{\xi}\phi U+\ep g(\xi, A\phi U)N
=\nabla_{\xi}\phi U=(\nabla_{\xi}\phi)U+\phi\nabla_{\xi}U\\
\overset{\eqref{structureEq}}{=}&\ep g(U,\xi)A\xi-\ep g(A\xi,U)\xi+\phi\nabla_{\xi}U
=-\ep g(U,U)\xi +\phi\nabla_{\xi}U.
\end{align*}
We may show that $\nabla_{\xi}U=0$. Its component in the direction of $\xi$ is given by $$g(\nabla_{\xi}U,\xi)=-g(U, \nabla_{\xi}\xi)=-g(U,\phi A\xi)=-g(U, \phi U)=0,$$
while the one in the direction of $U$ may be determined in the following way.  We use the Codazzi equation \eqref{cdz}  to evaluate
\begin{equation}
(\nabla_{\xi}A)U-(\nabla_UA)\xi=\ep\phi U,
\end{equation}
while using the definition for the covariant derivative of $A$ gives
\begin{align*}
(\nabla_{\xi}A)U-(\nabla_UA)\xi=&\nabla_{\xi}AU-A\nabla_{\xi}U-\nabla_UA\xi+A\nabla_U\xi\\
=&\nabla_{\xi}(\ep g(U,U)\xi)-\ep g(\nabla_{\xi}U,U)\xi-\nabla_UU+A\phi AU\\
=&2\ep g(\nabla_{\xi} U,U)\xi+\ep g(U,U)\phi A\xi-\ep g( \nabla_\xi U,U)\xi-\nabla_UU.
\end{align*}
We multiply by $\xi$ in the above two equations and obtain
\begin{align*}
0&=\ep g(\phi U, \xi) =\ep g(\nabla_{\xi}U,U) g(\xi,\xi)-g(\nabla_UU,\xi) \\ 
&=g(\nabla_{\xi}U,U) +g(U,\phi AU)=g(\nabla_{\xi}U,U). 
\end{align*}
Let us show now that $g(\nabla_{\xi}U, X)=0$, for any $X\perp U,\xi$ which is tangent to the hypersurface.
As before, we evaluate in two ways 
\begin{align*}
(\nabla_XA)\xi-(\nabla_{\xi}A)X\overset{\eqref{cdz}}{=}&-\ep \phi X\\
=&\nabla_XA\xi-A\nabla_X\xi-\nabla_{\xi}AX+A\nabla_{\xi}X\\
=&\nabla_XU-A\phi AX+A\nabla_{\xi}X,\\
=&\nabla_XU+A\nabla_{\xi}X,
\end{align*}
where we have used that $AX=0$ for $X\perp U$. We take the component in the direction of $\xi$ in the above relation and obtain
\begin{align*}
0&=-\ep g(\phi X,\xi)=g(\nabla_XU,\xi)+g(A\nabla_{\xi}X,\xi)\\
&=-g(U,\nabla_X\xi)+g(\nabla_{\xi}X,U) = -g(U,\phi AX)+g(\nabla_{\xi}X,U)\\
=&-g(X,\nabla_{\xi}U).
\end{align*}
Eventually, it follows that $\nabla_{\xi}U=0$ and therefore 
\begin{equation}
\tilde{\nabla}_{\xi}\tilde{\nabla}_{\xi}\xi=-\ep g(U,U)\xi.
\end{equation}
Next, we consider $\alpha:\de \to M$ the integral curve of $\xi$ such that  $\xi_{\alpha(s)}=\alpha'(s)$. We denote by $U_s:=U|_{\alpha(s)}$ and we may write
\begin{equation}\label{3alpha}
\tilde{\nabla}_{\alpha'(s)}\tilde{\nabla}_{\alpha'(s)}\alpha'(s)=-\ep g(U_s,U_s)\alpha'.
\end{equation}
\textbf{Case 1a)} Suppose first that $\alpha$ is a geodesic. It follows that $\tilde{\nabla}_{\alpha'}\alpha'=0$ and therefore $g(U_s,U_s)=0$. Since $U_s$ is not light-like, it must hold that $U_s=0$. This implies
\begin{equation}
\nabla_{\xi}\xi|_{\alpha(s)}=\phi U_s=0.
\end{equation}
\textbf{Case 1b)} Suppose  that $\alpha$ is not a geodesic. Along the curve $\alpha$ we have 
\begin{equation}\label{1bphius}
\tilde{\nabla}_{\alpha'}\alpha'=  \phi U_s\neq 0
\end{equation}
and since $U$ is not light-like, we can normalize $\phi U_s$, denoting:
\begin{gather*}
F_2(s):=\frac{1}{\sqrt{| g( U_s, U_s) |}}\phi U_s,\\
\ep_2:=\mathrm{sign}(g( U_s, U_s)),\ 
\ep_2 \kappa_1^2=g(U_s, U_s),
\end{gather*} where $\kappa_1> 0$.
For $F_1:=\alpha'$ and $F_2$ defined as above, we have 
\begin{equation}
\tilde{\nabla}_{F_1}F_1=\ep_2 \kappa_1F_2. \label{Fr1}
\end{equation}
We define $F_3$ as
\begin{align}
F_3= \tilde{\nabla}_{F_1}F_2+\ep_1 \kappa_1 F_1,\label{Fr2}
\end{align}
where $g(F_1,F_1)=\ep_1$ and  $F_3$ could be a light-like vector (see \cite{PhDthesis}).
 In fact, from \eqref{3alpha} and \eqref{Fr1} we obtain that 
 \begin{equation}
 \ep_2 \kappa_1'F_2-\ep_2\ep \kappa_1^2F_1+\ep_2 \kappa_1 F_3=-\ep g(U_s,U_s) F_1.
 \end{equation}
We need to check that $\{F_1,F_2,F_3\}$ are orthogonal. 
If we multiply by $F_1$ in \eqref{Fr1}, 
\begin{align*}
\ep_2 \kappa_1 g(F_2,F_1)&=g(\tilde{\nabla}_{F_1}F_1,F_1) =
\frac{1}{2}\frac{d}{ds}g(F_1,F_1)=0, \\0&=g(F_2,F_1).
\end{align*}
If we multiply with $F_1$ in \eqref{Fr2}, we obtain
\begin{align*}
g(\tilde{\nabla}_{F_1}F_2,F_1)&=-\ep_1\kappa_1g(F_1,F_1)+g(F_3,F_1) \Leftrightarrow\\
-g(F_2,\tilde{\nabla}_{F_1}F_1)&=-\kappa_1+g(F_3,F_1) \Leftrightarrow\\
-g(F_2,\ep_2 \kappa_1 F_2)&=-\kappa_1+g(F_3,F_1) \Leftrightarrow\\
0&=g(F_3,F_1).
\end{align*}
Similarly, it follows that $g(F_3,F_2)=0$ when taking the inner product with $F_2$ in \eqref{Fr2}. From this fact, and  as $\ep_2\kappa_1^2=g(U_s,U_s)$, $\kappa_1\neq 0$, we obtain that $ \kappa_1'=0$ and $F_3=0$.   This shows that $\alpha$ is a curve of order $2$ and $\kappa_1$ is constant. In other words, $\alpha$ is a circle. 
Moreover, we want to compute its holomorphic torsion, defined in \cite{A-O'M}, $\tau=g(F_1, \phi F_2)$. It follows directly from \eqref{1bphius} and \eqref{Fr1}:
\begin{align*}
\tau&=g\left(\xi_{\alpha(s)}, \frac{1}{\ep_2 \kappa_1}\phi^2 U_s\right)
=\frac{1}{\ep_2 \kappa_1} g(\xi_{\alpha(s)}, -U_s+\ep \eta(U_s)\xi_{\alpha(s)} )=0.
\end{align*}
Therefore, we conclude that $\alpha$ is a totally real  circle. 

Moreover, we will prove that any totally real circle in $\cpn$ is actually contained in a totally geodesic and totally real surface $\Sigma\subset \cpn$, as follows. Let $\alpha:I\to \cpn$ denote an arc length, totally real circle in $\cpn$. Let $t_0\in I$,  $p= \alpha(t_0)$ and {$F_1(t_0),F_2(t_0)\in T_{\alpha(t_o)}\cpn$, which span a totally real plane}. Since $\alpha$ is a Frenet curve, then $F_1$ and $F_2$ are never light-like. By Lemma \ref{totgeosurf}, there exists a totally real and totally geodesic surface $\Sigma$ in $\cpn$ such that $p\in\Sigma$ and  $T_p\Sigma=\mathrm{span}\{ F_1(t_0), F_2(t_0) \} $. By Definition \ref{totrealcircle}, 
\begin{align}\label{alpha1}
\begin{split}
\tilde \nabla_{\alpha'(t) }\alpha'(t)=\ep_2 \kappa_1 E_2,\\
\tilde \nabla_{\alpha'(t) } E_2=-\ep_1\kappa_1 E_1,
\end{split}
\end{align}
where $\tilde \nabla $ is the Levi-Civita connection on $\cpn$ and $\alpha'(t)=E_1$. 
Consider now a curve $\beta: I\to \Sigma$ contained in the surface $\Sigma$ in $\cpn$, such that $p=\beta(t_0)$ and $\beta$ satisfies the following Frenet system of equations in $\Sigma$:
\begin{align*}
\nabla_{\beta'}\beta'=\ep_2 \kappa_1 E_2,\quad 
\nabla_{\beta'} E_2=-\ep_1\kappa_1 E_1,
\end{align*}
where $\nabla$ is the Levi-Civita connection on $\Sigma$. We impose, additionally, that $\alpha$ and $\beta$ satisfy the same initial conditions, $\alpha(t_0)=\beta(t_0)$, $\beta'(t_0)=F_1(t_0)=E_1(t_0)$ and $E_2(t_0)=F_2(t_0)$. Since $\Sigma$ is totally geodesic in $\cpn$, i.e. $\tilde{\nabla}_XY=\nabla_XY$ for any  vectors $X,Y$ tangent to $\Sigma\subset\cpn$,  we have that the Frenet equations of $\beta$ are the same both in $\Sigma$ and in $\cpn$. Therefore, $\beta$ also satisfies \eqref{alpha1}. By recalling the uniqueness of solution to initial value problems, then $\alpha=\beta$.  That is, $\alpha$ is contained in the totally geodesic, totally real surface $\Sigma$. 
The type of surface will depend on the causal character of the Frenet system $\{F_1=\alpha',F_2\}$ of $\alpha$. To explain this, we resort to Lemma \ref{totgeosurf}. Take $\Pi=\mathrm{span}\{F_1,F_2\}$ at a given point. 
\begin{enumerate}
\item If $\Pi$ is space-like, then the surface is (an open subset of) $\mathbb{R}P^2$. 
\item If $\Pi$ is an indefinite plane, the surface is locally isometric to $\mathbf{S}^2_1$.
\item If $\Pi$ is negative definite, the surface is $\mathbb{H}^2_2$. 
\end{enumerate}

\noindent\textbf{Case 2.} Assume $U$ is light-like. Again, we have 
\begin{equation}\label{phiu}
\tilde{\nabla}_{\xi}\xi=\nabla_{\xi}\xi+\ep g(A\xi,\xi)N=\phi A\xi=\phi U,
\end{equation} from which
\begin{align}
\tilde{\nabla}_{\xi}\tilde{\nabla}_{\xi}\xi	=& \tilde{\nabla}_\xi\phi U 
		=\nabla_{\xi}\phi U+\ep g(U, \phi U)N 
		=\nabla_{\xi}\phi U 
=(\nabla_{\xi}\phi)U+\phi\nabla_{\xi}U\nonumber\\
\overset{\eqref{structureEq}}{=}& \ep g(U,\xi)A\xi-\ep g(A\xi,U)\xi+\phi\nabla_{\xi}U 
					=\phi \nabla_{\xi} U.\label{phinablaxiu}
\end{align}
Further on, we show that $\nabla_{\xi}U=0$. Indeed, firstly we have  
\begin{align*}
g(\nabla_{\xi}U,\xi)&=-g(U,\nabla_{\xi}\xi)=-g(U, \phi U)=0,\\
g(\nabla_{\xi}U,U)&=\frac{1}{2} \xi(g(U,U))=0. 
\end{align*}
and we may choose $X\in \dd$, such that $g(X,U)=\ep$. This gives immediately that $AX=\ep g(AX,\xi)\xi=\ep g(X,U)\xi=\xi$. Next, we compute in two different ways $(\nabla_XA)\xi-(\nabla_{\xi}A)X$:
\begin{align*}
(\nabla_XA)\xi-(\nabla_{\xi}A)X&=-\ep\phi X\\
			&=\nabla_XA\xi-A\nabla_X\xi-(\nabla_{\xi}AX-A\nabla_{\xi}X)\\
			&=\nabla_XU-A\phi AX-\nabla_{\xi}\xi+A\nabla_{\xi}X.
\end{align*}
We take the components in the direction of $\xi$ and obtain that 
$$0=g(\nabla_XU,\xi)+g(\nabla_{\xi}X,U)=-g(X,\nabla_{\xi}U).$$
We conclude therefore that $\nabla_{\xi}U=0$ and so, equation \eqref{phinablaxiu} implies that $\tilde{\nabla}_{\xi}\tilde{\nabla}_{\xi}\xi=0$. \\
As before, we consider the integral curve $\alpha$ of $\xi$, that is $\alpha:\de\to M$ and  $\alpha'(s) = \xi |_{ \alpha(s)}$, for the {arc parameter $s$}. We will use the index $``s"$ for vector fields along the curve $\alpha$.  By \eqref{phiu}, we obtain that $\alpha$ is determined by the following system of equations:
\begin{gather}\label{Cauchy}
F_1=\alpha'=\xi_{s}, \  
 \tilde{\nabla}_{F_1}F_1=\phi U_{s}=:F_2, \  \tilde{\nabla}_{F_1}F_2=0, \\
 g(F_1,F_2)=0, \   g(F_1,J F_2)=0, \nonumber
\end{gather}
where $F_2$ is a light-like vector. 

Let us show that  the curve $\alpha$ is  contained in a $3$-dimensional submanifold in $\cpn$. 
We have that $F_1=\alpha'=\xi_{s}$, $JF_1=N_{s}$ and $0\neq F_2\in \mathrm{span}\{F_1,JF_1\}^{\perp}=\dd$. There exists a local orthonormal basis of $\dd$ of the form $\{w_1,Jw_1,\ldots, w_{n-1}, Jw_{n-1}\}$ such that the first $t$ vectors $w_1,\ldots, w_t$ are time-like and the remaining vectors $w_{t+1},\ldots, w_{n-1}$ are space-like. There exist $a_i,b_i$ functions for $i=1,\ldots, n-1$, such that we may write $F_2=\sum_{k=1}^{n-1}(a_kw_k+b_kJw_k)$, If we denote by $v_k:=a_kw_k+b_kJw_k$, this is equivalent to $F_2=\sum_k v_k$. Moreover, the time-like vector $E_1=\sum_{k=1}^tv_k$ and the space-like vector $E_2=\sum_{k=t+1}^{n-1} v_k$ {satisfy $F_2=E_1+E_2$ and $g(E_1,JE_2)=0$.} Given  the totally real $Q:=\mathrm{span}\{F_1(0), E_1,E_2\}\subset \dd_{\alpha(0)}$, there exists a totally geodesic, totally real, $3$-dimensional submanifold $B$ of $\cpn$, for which $\alpha(0)\in B$, $F_2(0)\in Q$ and $T_{\alpha(0)}B=Q$. Given the uniqueness of solution for the Cauchy problem in \eqref{Cauchy}, we have that the curve $\alpha$ is entirely contained in $B$.

To describe $B$, we resort to Lemma \ref{3submanifolds}. To make it more simple and understandable, we work with the previous descriptions to Lemma \ref{3submanifolds}. 

\noindent\textbf{Case c-1)} $\mathbf{B=dS^3}$. Let $\bar{D}$ and $D$ the Levi-Civita connections of $\mathbb{L}^4$ and $dS^3$, respectively. Bearing in mind that the position vector $\chi:dS^3\to \mathbb{L}^4$ is a unit space-like normal vector field, the Gauss equation is
\[ \bar{D}_XY=D_XY-\langle X,Y\rangle \chi, 
\]
for any $X,Y\in \Gamma(TdS^3)$, where $\langle,\rangle$ is the standard flat metric in $\mathbb{L}^4$.

Let $\alpha:I\to dS^3$ be a unit, space-like curve in the de Sitter 3-space, with $F_1=\alpha'$. By combining the Gauss equation and  \eqref{curve-c}, we obtain the following ODE,
\[ F_2=D_{F_1}F_1 = \bar{D}_{F_1}F_1+\langle\alpha',\alpha'\rangle\alpha = 
\alpha''+\alpha, \quad 0=D_{F_1}F_2=\alpha'''+\alpha'.
\]
As a consequence, the vector $F_2$ along $\alpha$ must be constant, light-like, and orthogonal to $\alpha'$. Thus, everything reduces to the following initial value problem: 
\[ \alpha''+\alpha=F_2, \ \alpha(0)=p_0, \ \alpha'(0)=v_0, 
\quad p_0,v_0,F_2\in \mathbb{L}^4, \]
with  $\{p_0, v_0, F_2\}$ orthogonal, $\{p_0,v_0\}$ unit space-like, and $F_2$ light-like. The solution is
$\alpha:\RR\to dS^3$, 
\[ \alpha(s)=F_2+\cos(s)(p_0-F_2)+\sin(s)v_0.
\]
Indeed, $\langle \alpha(s),\alpha(s)\rangle=1$, 
$\langle\alpha'(s),\alpha'(s)\rangle =1$, and $\alpha''(s)+\alpha(s)=F_2$. 

Suppose that $\alpha:I\to dS^3$ could be unit, time-like. As before, at a certain point, we would reach to a set of vectors $p_0=\alpha(s_0)$ unit space-like, $v_0=\alpha'(s_0)$ unit time-like, $p_0\perp v_0$. However, if $v_0$ is time-like, since $\dim dS^3=3$ and its index is $1$, then $F_2\in v_0^{\perp}$ cannot be light-like. \\

\noindent\textbf{Case c-2)} $\mathbf{B=S_2^2}$. According to Lemma \ref{3submanifolds}, we consider the hyperquadric $\hat{B}_2=\{(x_1,\ldots,x_4)\in\RR^4_2 : -x_1^2-x_2^2+x_3^2+x_4^2=1\}$, which is locally anti-isometric to the anti-de Sitter 3-space. 
Let $\bar{D}$ and $D$ the Levi-Civita connections of $\mathbb{R}^4_2$ and $\hat{B}_2$, respectively. Bearing in mind that the position vector $\chi:\hat{B}_2\to \mathbb{R}^4_2$ is a unit space-like normal vector field, the Gauss equation is
\[ \bar{D}_XY=D_XY-\langle X,Y\rangle \chi, 
\]
for any $X,Y\in \Gamma(T\hat{B}_2)$, where $\langle,\rangle$ is the standard flat metric in $\mathbb{R}^4_2$. 

Let $\alpha:I\to \hat{B}_2$ be a unit, time-like curve in $\hat{B}_2$, with $F_1=\alpha'$. By combining the Gauss equation and  \eqref{curve-c}, we obtain the following ODE,
\[ F_2=D_{F_1}F_1 = \bar{D}_{F_1}F_1+\langle\alpha',\alpha'\rangle\alpha = 
\alpha''-\alpha, \quad 0=D_{F_1}F_2=\alpha'''-\alpha'.
\]
As a consequence, the vector $F_2$ along $\alpha$ must be constant, light-like, and orthogonal to $\alpha'$. Thus, everything reduces to the following initial value problem: 
\[ \alpha''-\alpha=F_2, \ \alpha(0)=p_0, \ \alpha'(0)=v_0, 
\quad p_0,v_0,F_2\in \mathbb{L}^4, \]
with orthogonal $\{p_0, v_0, F_2\}$, $p_0$ unit space-like, $v_0$ unit time-like, and $F_2$ light-like. The solution is 
\[\alpha:\RR\to \hat{B}_2, \quad \alpha(s)=\cosh(s)(p_0+F_2)+\sinh(s)v_0-F_2.
\]

Suppose that $\alpha:I\to \hat{B}_2$ could be unit, space-like. As before, at a certain point, we would reach to a set of vectors $p_0=\alpha(s_0)$ unit space-like, $v_0=\alpha'(s_0)$ unit space-like. However, since $\dim \hat{B}_2$ and the index is 2, if $v_0$ is space-like, then $F_2\in v_0^{\perp}$ cannot be light-like. 
\end{proof} 	

\section{Examples}\label{examples}
We develop the first example in detail, and simplify the rest due to the similarities.
\begin{example}\normalfont \label{exam-first}
We take $1\leq p\leq n-2$,  $n\geq 3$. Consider $\Omega=\{z=(z_1,\ldots,z_n)\in \mathbb{S}_{2p}^{2n-1} : z_n\neq 0\}$. We define the map 
\[\hat\psi:\hat{M}=\mathbb{R}\times \Omega\to \mathbb{S}^{2n+1}_{2p}, \quad \hat \psi(t,z)=(z_1,\ldots,z_{n-1},\cos(t)z_n,\sin(t)z_n).\] 
Since $p\leq n-2$, then $n+1-p>2$, so that 
\[\mathbf{A}_{t}=\begin{pmatrix} I_{n-1} & 0 & 0 \\ 0 & \cos(t) & -\sin(t) \\ 0 & \sin(t) & \cos(t)
\end{pmatrix} \in SU(p,n+1-p),
\]
and so, it represents an isometry of $\sn$, which also induces an isometry of $\cpn$. Define for each $t\in\mathbb{R}$, $\Omega^t:= \hp\left(\{t\}\times\Omega\right)$. Thus, $\mathbf{A}_t(\hp(\{0\}\times \Omega))=\mathbf{A}_t(\Omega^0)$. Each $\Omega^t$, $t\in\RR$, is totally geodesic, and so $\pi(\Omega^t)\subset\cpn$ is an open subset of a totally geodesic complex hyperplane. With them, we construct the following ruled real hypersurface in $\cpn$, 
	\[\psi: M=\RR\times \pi(\Omega) \to \cpn, \quad \psi(t,[z])=\pi(\hp(t,z)), \quad 
	\psi(M)=\bigcup_{t\in\mathbb{R} }\pi(\Omega^t).
	\] 
The differential of $\hat\psi$ is
\begin{align*} \hat\psi_t=\partial_t\hat\psi & =(0,\ldots,0,-\sin(t)z_n,\cos(t)z_n), \\ 
d\hat\psi(0,X) & = (X_1,\ldots,X_{n-1},\cos(t)X_n,\sin(t)X_n), \quad X\in T_z\Omega. 
\end{align*}
If we recall $iz=J\chi(z)$, the vertical part of $\pi:\sn\to \cpn$ is $\mathcal{V}_z=\mathrm{Span}\{iz\}=\mathrm{Span}\mathrm\{J\chi(z)\}$. Then, 
\[ J\chi\vert_{\hat\psi(t,z)} = i\hat\psi(t,z) = d\hat\psi_{(t,z)}(0,iz).\]
Given $(t,z)\in \mathbb{R}\times\Omega$, and $X\in T_z\Omega$, then $\langle z,X\rangle=0$. 
Simple computations lead to
\begin{align*}
 \langle i\hat\psi,d\hat\psi(0,X)\rangle = \langle iz,X\rangle, & \quad \langle \partial_t\hat\psi,d\hat\psi(0,X)\rangle = 0, \\
 \langle d\hat\psi(0,X),d\hat\psi(0,Y)\rangle = \langle X,Y \rangle, & \quad 
\langle \partial_t\hat\psi,\partial_t\hat\psi\rangle = \vert z_n\vert^2>0. 
\end{align*}
According to this, we define
\[ \hat{\xi}_{\hat\psi(t,z)}=\frac{\partial_t\hat\psi}{\vert z_n\vert}, \quad 
\hat{N}_{\hat\psi(t,z)} = i\hat{\xi}_{\hat\psi(t,z)}=\frac{1}{\vert z_n\vert}\left(0,\ldots,0,-i \sin(t)z_n, i \cos(t)z_n\right).
\]
Let us see that $\hat{N}$ is a horizontal, space-like, unit normal vector field along $\hat\psi$:
\begin{align*}
& \langle \hat{N},\partial_t\hat\psi\rangle = \frac{1}{\vert z_n\vert} \langle i \partial_t\hat\psi,\partial_t\hat\psi\rangle =0, 
\langle \hat{N},\hat{N}\rangle = \frac{1}{\vert z_n\vert^2} \langle i\partial_t\hat\psi,i\partial_t\hat\psi\rangle =1, \\
& \langle \hat{N},d\hat\psi(0,X)\rangle = \frac{1}{\vert z_n\vert} \mathrm{Re}\left(
-\sin(t)i z_n \cos(t) \overline{X_n}+\cos(t)i z_n\sin(t)\overline{X_n}\right)=0, \\
& \langle \hat{N},i\hat\psi\rangle = \langle i\hat{\xi},i\hat\psi\rangle =\langle \hat{\xi},\hat\psi\rangle = 
\frac{1}{\vert z_n\vert}\langle \partial_t\hat\psi,\hat\psi\rangle =0. \end{align*}
We point out that an integral curve of $\hat{\xi}$ is
\begin{equation} \label{integralxi} \alpha(s)=\hat\psi\left(t+s/\vert z_n\vert, z\right)
= \left( z_1,\ldots,z_{n-1},\cos\left(t+s/\vert z_n\vert\right)z_n,\sin\left(t+s/\vert z_n\vert\right)z_n\right).
\end{equation}
Indeed, $\alpha(0)=\hat\psi(t,z)$, $\alpha'(0)=
\left. \frac{d}{ds}\hat\psi\left(t+s/\vert z_n\vert,z\right) \right\vert_{s=0} =
\frac{1}{\vert z_n\vert} \partial_t\hat\psi (t,z)=\hat{\xi}_{\hat\psi(t,z)}$, and 
\[\alpha'(s)=\frac{1}{\vert z_n\vert}\left(0,\ldots,0,-\sin\left(t+s/\vert z_n\vert\right)z_n,\cos\left(t+s/\vert z_n\vert\right)z_n\right) = \hat{\xi}_{\hat\psi(t+s/\vert z_n\vert,z)}= \hat{\xi}_{\alpha(s)}.
\]
Let $\hat{A}$ be the shape operator associated with $\hat{N}$. By \eqref{GaussEq}, we obtain 
\begin{align*}
&\hat{A}\hat{\xi}=-\hat{\nabla}_{\hat{\xi}}\hat{N} = - D_{\hat{\xi}}\hat{N} = - \left.\frac{d}{ds} \hat{N}_{\alpha(s)}\right \vert_{s=0} = - \left.\frac{d}{ds} \hat{N}_{\hat\psi(t+s/\vert z_n\vert,z)}\right \vert_{s=0} \\
& = - \left.
\frac{d}{ds} \frac{i}{\vert z_n\vert} \left(0,\ldots,0,-\sin\left(t+s/\vert z_n\vert\right)z_n,
\cos\left(t+s/\vert z_n\vert\right)z_n\right) 
\right \vert_{s=0} \\
&=\frac{1}{\vert z_n\vert^2} \left(0,\ldots,0, i\cos(t)z_n,i \sin(t)z_n\right). 
\end{align*}
Finally, 
\begin{align*}\langle \hat{A}\hat{\xi},\hat{\xi}\rangle &= \frac{1}{\vert z_n\vert^3}
\langle (0,i \cos(t)z_n,i \sin(t)z_n), (0,-\sin(t)z_n,\cos(t)z_n)\rangle  \\
& = 
\frac{1}{\vert z_n\vert^3}\mathrm{Re}\left( -i \cos(t)\sin(t)\vert z_n\vert^2+i\sin(t)\cos(t)\vert z_n\vert^2\right)=0.
\end{align*}
Now, since a unit normal vector field to $M$ is $N=\pi_*(\hat{N})$, we also have $\xi=\pi_*(\hat{\xi})$, and the shape operator $A$ of $M$ is $A=\pi_*\hat A$, as expected. With all this information, we obtain $g(A\xi,\xi)=0$. As  $g(AX,Y)=0$ for any $X,Y\perp \xi$, $X,Y\in TM$, then $M$ is minimal. 

Let us study now the integral curves of $\hat{\xi}$. Note that $\langle \alpha',i\alpha\rangle=0$, which means that $\alpha$ is a horizontal curve. By \eqref{GaussEq}, 
\begin{align*}
 F(s):= & \hat{\nabla}_{\alpha'(s)}\alpha'(s) = D_{\alpha'(s)}\alpha'(s)+\langle \alpha'(s),\alpha'(s)\rangle \alpha(s) = \alpha''(s)+\alpha(s) \\
& = \frac{1}{\vert z_n\vert^2}(0,\ldots,0,-\cos(t+s/\vert z_n\vert)z_n,-\sin(t+\vert z_n\vert)z_n)+\alpha(s) \\
& = \left(z_1,\ldots,z_{n-1}, \left(1-\frac{1}{\vert z_n\vert^2}\right) \cos(t+s/\vert z_n\vert)z_n,
\left(1-\frac{1}{\vert z_n\vert^2}\right) \sin(t+s/\vert z_n\vert)z_n\right); \\
 \langle F(s),F(s)\rangle &= -\sum_{j=1}^p\vert z_j\vert^2+\sum_{j=p+1}^{n-1}\vert z_j\vert^2 
+\left(1-\frac{1}{\vert z_n\vert^2}\right)^2\vert z_n\vert^2 \\
&= 1+\left[\left(1-\frac{1}{\vert z_n\vert^2}\right)^2-1\right]\vert z_n\vert^2
 = \frac{1}{\vert z_n\vert^2}-1.
\end{align*}
When $\vert z_n\vert^2=1$,  $F(s)$ is zero or light-like. The first case holds when $z=(0,\ldots,0,e^{ir})$, and the second one, otherwise. When $\vert z_n\vert^2<1$, $F(s)$ is always space-like. If $\vert z_n\vert^2>1$,  $F(s)$ is always time-like. In addition, similar computations give $\langle F(s),\alpha(s)\rangle=\langle F(s),i\alpha(s)\rangle=0$, which show that $F(s)$ is a horizontal vector along $\alpha$. We need this to project it safely to $\cpn$. 

\noindent $\bullet$ Case $\vert z_n\vert=1$. Then, $F(s)=(z_1,\ldots,z_{n-1},0,0)$, which is constant and either zero or light-like. The curve $\pi(\alpha(s))$ is the case $a)$ or $c)$ of Theorem \ref{TH1}. 

\noindent $\bullet$ Case $\vert z_n\vert<1$. Then, $F(s)$ is space-like. To compute the Frenet system, we note that $1-1/\vert z_n\vert^2$ does not depend on $s$. We obtain 
\begin{align*}
&\kappa_1:=\sqrt{\frac{1}{\vert z_n\vert^2}-1}>0, \quad F_1(s)=\alpha'(s),\\
& F_2(s):=\frac{F(s)}{\kappa_1}
=\left(\frac{z_1}{\kappa_1},\ldots,\frac{z_{n-1}}{\kappa_1},-\kappa_1\cos(s/\vert z_n\vert)z_n,-\kappa_1(s)\sin(s/\vert z_n\vert)z_n\right), \\
& \hat{\nabla}_{\alpha'(s)}\alpha'(s)=\kappa_1 F_2(s). 
\end{align*}
Next, since $\langle F_2,\alpha'\rangle=0$, by \eqref{GaussEq},
\begin{align*}
\hat{\nabla}_{\alpha'(s)}F_2(s)=F_2'(s)-\langle \alpha'(s),F_2(s)\rangle\alpha(s)=-\kappa_1\alpha'(s).
\end{align*}
This is a Frenet curve of order 2, with $F_1,F_2$ space-like, with constant curvature $\kappa_1$ and 
\begin{align*}
&  \langle JF_1(s),F_2(s)\rangle = \frac{1}{\vert z_n\vert}\left\langle (0,\ldots,0,-i\sin(s/\vert z_n\vert)z_n,
i\cos(s/\vert z_n\vert)z_n),F_2(s)\right\rangle =0.
\end{align*}
Thus, if we project to $\cpn$, $\pi(\alpha)$ is a curve of type $b)$ in Theorem \ref{TH1}.

\noindent $\bullet$ Case $\vert z_n\vert>1$. Then, $F(s)$ is time-like. Similar computations give 
\begin{align*}
&\kappa_1:=\sqrt{1-\frac{1}{\vert z_n\vert^2}}>0, \quad F_1(s)=\alpha'(s),\\
& F_2(s):=\frac{-F(s)}{\kappa_1}
=\left(\frac{-z_1}{\kappa_1},\ldots,\frac{-z_{n-1}}{\kappa_1},-\kappa_1\cos(s/\vert z_n\vert)z_n,-\kappa_1\sin(s/\vert z_n\vert)z_n\right), \\
& \hat{\nabla}_{\alpha'(s)}\alpha'(s)=-\kappa_1 F_2(s). 
\end{align*}
We also obtain $\langle JF_1(s),F_2(s)\rangle=0$. Next, since $\langle F_2,\alpha'\rangle=0$, by \eqref{GaussEq},
\begin{align*}
\hat{\nabla}_{\alpha'(s)}F_2(s)=F_2'(s)-\langle \alpha'(s),F_2(s)\rangle\alpha(s)=-\kappa_1\alpha'(s).
\end{align*}
If we project to $\cpn$, $\pi(\alpha)$ is a curve of type $b)$ in Theorem \ref{TH1}. 

We define the function $f:(0,\pi/2]\to\mathbb{R}$, $f(r)=\sqrt{2}\sin(r)$. Clearly, for $0<r<\pi/4$, $0<f(r)<1$; $f(\pi/4)=1$; and for $\pi/4<r\leq \pi/2$, $f(r)>1$. With this, we consider the  curve $\gamma:(0,\pi/2]\to\Omega\subset  \mathbb{S}^{2n-1}_{2p}$, $\gamma(r)=(1,0,\ldots,0,\sqrt{2}\cos(r),\sqrt{2}\sin(r))$. In particular, the points $\gamma(\pi/8)$, $\gamma(\pi/4)$ and $\gamma(\pi/2)$ lie in the same connected component of $\Omega$, satisfying $\vert z_n\vert<1$, $\vert z_n\vert=1$ and $\vert z_n\vert=\sqrt{2}$, respectively. Thus, there exist different points lying in the same connected component of $\hat{M}$, such that the integral curves of $\hat{\xi}$ passing through them behave in a very different way. 

By projecting everything to $\cpn$, we have a  minimal ruled real hypersurface in $\cpn$ such that certain integral curves of $\xi$, starting at different points at the same connected component, behave as in any case of Theorem 2. 
\end{example}

\begin{example}\label{examp-third} 
\normalfont 
We take $1\leq p\leq n-2$, $n\geq 4$. Consider $\Omega=\{z=(z_1,\ldots,z_n)\in \mathbb{S}_{2p}^{2n-1} : z_1\neq 0\}$. We define the map 
\[\hp:\hat{M}:=\mathbb{R}\times \Omega\to \mathbb{S}^{2n+1}_{2p}, \quad 
\hp(t,z)=
(\cosh(t)z_1,z_2\ldots,z_{n},\sinh(t)z_1). \] 
Again, the matrix 
\[\mathbf{A}_{t}=\begin{pmatrix} \cosh(t) & 0 & \sinh(t) \\ 0 & I_{n-1} & 0 \\ \sinh(t) & 0 & \cosh(t)
\end{pmatrix} \in SU(p,n+1-p),
\]
and it represents an isometry of $\sn$, inducing an isometry of $\cpn$. Define for each $t\in\mathbb{R}$, $\Omega^t:= \hp\left(\{t\}\times\Omega\right)=\mathbf{A}_t(\hp(\{0\}\times \Omega))=\mathbf{A}_t(\Omega^0)$. Each $\Omega^t$, $t\in\RR$, is totally geodesic, and so $\pi(\Omega^t)\subset\cpn$ is a totally geodesic complex hyperplane. Then, we construct the following ruled real hypersurface in $\cpn$, 
	\[\psi: M=\RR\times \pi(\Omega) \to \cpn, \quad \psi(t,[z])=\pi(\hp(t,z)), \quad 
	\psi(M)=\bigcup_{t\in\mathbb{R} }\pi(\Omega^t).
	\] 
The differential of $\hp$ is
	\begin{align*} \partial_t\hp & =(\sinh(t)z_1,0,\ldots,0,\cosh(t)z_1), \\ 
		d\hp(0,X) & = (\cosh(t)X_1,X_2\ldots,X_{n},\sinh(t)X_1), \quad X\in T_z\Omega.
	\end{align*}
	Similarly to the previous example,
	\begin{align*}
		&\langle i\hp,d\hp(0,X)\rangle = \langle iz,X\rangle,  \quad \langle \partial_t\hp,d\hp(0,X)\rangle = 0, \\
		&  \langle d\hp(0,X),d\hp(0,Y)\rangle = \langle X,Y \rangle,  \quad 
		\langle \partial_t\hp,\partial_t\hp\rangle = \vert z_1\vert^2>0, \\
		& \hat{\xi}_{\hp(t,z)}=\frac{\partial_t\psi}{\vert z_1\vert}, \quad 
		\hat{N}_{\hp(t,z)} = i\hat{\xi}_{\hp(t,z)}=\frac{i}{\vert z_1\vert}\left(\sinh(t)z_1,0,\ldots,0,\cosh(t)z_1\right).
	\end{align*}
	Now, $\hat{N}$ is a horizontal, space-like, unit normal vector field along $\hp$. An integral curve of $\hat{\xi}$ is
	\begin{equation} \label{integralxi} \alpha(s)=\hp\left(t+s/\vert z_1\vert, z\right)
	= \left( \cosh\left(t+s/\vert z_1\vert\right)z_1,z_2,\ldots,z_{n},\sinh\left(t+s/\vert z_1\vert\right)z_1\right).
	\end{equation}
	Let $\hat{A}$ be the shape operator associated with $\hat{N}$. Then, by \eqref{GaussEq}, 
	\begin{align*}
		&\hat{A}\hat{\xi}=-\hat{\nabla}_{\hat{\xi}}\hat{N} = - D_{\hat{\xi}}\hat{N} = - \left.\frac{d}{ds} \hat{N}_{\alpha(s)}\right \vert_{s=0} = - \left.\frac{d}{ds} \hat{N}_{\hp(t+s/\vert z_1\vert,z)}\right \vert_{s=0} \\
		& = - \left.
		\frac{d}{ds} \frac{i}{\vert z_1\vert} \left(\sinh\left(t+s/\vert z_1\vert\right)z_1,0,\ldots,0,\cosh\left(t+s/\vert z_1\vert\right)z_1\right) 
		\right \vert_{s=0} \\
		&=\frac{-i}{\vert z_1\vert^2} \left(\cosh(t)z_1,0,\ldots,0, \sinh(t)z_1\right); \\
		& \langle \hat{A}\hat{\xi},\hat{\xi}\rangle=0. 
	\end{align*}
	Again, $\psi(M)$ is a minimal ruled real hypersurface in $\cpn$. 
	
Let us study now the integral curves of $\hat{\xi}$. Note that $\langle \alpha',i\alpha\rangle=0$, which means that $\alpha$ is a horizontal curve. By \eqref{GaussEq}, 
	\begin{align*}
		&F(s):= \hat{\nabla}_{\alpha'(s)}\alpha'(s) = D_{\alpha'(s)}\alpha'(s)+\langle \alpha'(s),\alpha'(s)\rangle \alpha(s) =
		\alpha''(s)+\alpha(s) \\
		& = \frac{1}{\vert z_1\vert^2}(\cosh\left(t+s/\vert z_1\vert\right),0,\ldots,0,\sinh(t+s/\vert z_1\vert)z_1)+\alpha(s) \\
		& = \left(\left(1+\frac{1}{\vert z_1\vert^2}\right)\cosh(t+s/\vert z_1\vert)z_1,z_2,\ldots,z_{n}, \left(1+\frac{1}{\vert z_1\vert^2}\right) \sinh(t+s/\vert z_1\vert)z_1\right); \\
		& \langle F(s),F(s)\rangle = -1.
	\end{align*}
By now, the Frenet system is 
	\begin{align*}
		&\kappa_1:=\sqrt{\frac{1}{\vert z_1\vert^2}+1}>0, \quad F_1(s)=\alpha'(s),\\
		& F_2(s):=\frac{-F(s)}{\kappa_1}
		= \left(-\kappa_1\cosh(t+s/\vert z_1\vert)z_1,\frac{-z_2}{\kappa_1},\ldots,\frac{-z_{n}}{\kappa_1},-\kappa_1 \sinh(t+s/\vert z_1\vert)z_1\right), \\
		& \varepsilon_1=\langle \alpha',\alpha'\rangle=+1,\ \varepsilon_2=\langle F_2,F_2\rangle=-1, \quad 
		\hat{\nabla}_{F_1(s)}F_1(s)=-\kappa_1 F_2(s)=\varepsilon_2\kappa_1F_2(s). 
	\end{align*}
	
	We also obtain $\langle JF_1(s),F_2(s)\rangle=0$. Next, since $\langle F_2,\alpha'\rangle=0$, by \eqref{GaussEq},
	\begin{align*}
		\hat{\nabla}_{F_1(s)}F_2(s)=F_2'(s)+\langle F_1(s),F_2(s)\rangle\alpha(s)=
		-\kappa_1F_1(s).
	\end{align*}
	If we project to $\cpn$, $\pi(\alpha)$ is a curve of type $b)$ in Theorem \ref{TH1}. 
\end{example}

\begin{example}\label{examp-fourth}\normalfont 
We take $2\leq p\leq n-1$, $n\geq 3$, and $\Omega=\{z=(z_1,\ldots,z_n)\in \mathbb{S}_{2(p-1)}^{2n-1} : z_n	\neq 0\}$. We define the map 
\[\hp:\hat{M}=\mathbb{R}\times \Omega\to \mathbb{S}^{2n+1}_{2p}, \quad 
\hp(t,z)=
(\sinh(t)z_n,z_1,\ldots,z_{n-1},\cosh(t)z_n). \]  
We use $\mathbf{A}_{t}$ as in Example \ref{examp-third}. Define for each $t\in\mathbb{R}$, $\Omega^t:= \hp\left(\{t\}\times\Omega\right)=\mathbf{A}_{t}(\hp(\{0\}\times \Omega))=\mathbf{A}_{t}(\Omega^0)$. Again, 
$\Omega^t$ is also totally geodesic, and so  $\pi(\Omega^t)\subset\cpn$ is a totally geodesic complex hyperplane. Then, we construct the following ruled real hypersurface 
	\[\psi: M=\RR\times \pi(\Omega) \to \cpn, \quad \psi(t,[z])=\pi(\hp(t,z)), \quad 
	\psi(M)=\bigcup_{t\in\mathbb{R} }\pi(\Omega^t).
	\] 
The differential of $\hp$ is
	\begin{align*} \partial_t\hp & =(\cosh(t)z_n,0,\ldots,0,\sinh(t)z_n), \\ 
		d\hp(0,X) & = (\sinh(t)X_n,X_1\ldots,X_{n-1},\cosh(t)X_n), \quad X\in T_z\Omega.
	\end{align*}
Similarly to previous examples,
	\begin{align*}
		&\langle i\hp,d\hp(0,X)\rangle = \langle iz,X\rangle,  \quad \langle \partial_t\hp,d\hp(0,X)\rangle = 0, \\
		&  \langle d\hp(0,X),d\hp(0,Y)\rangle = \langle X,Y \rangle,  \quad 
		\langle \partial_t\hp,\partial_t\hp\rangle =- \vert z_n\vert^2<0, \\
		& \hat{\xi}_{\hp(t,z)}=\frac{\partial_t\psi}{\vert z_n\vert}, \quad 
		\hat{N}_{\hp(t,z)} = i\hat{\xi}_{\hp(t,z)}=\frac{i}{\vert z_n\vert}\left(\cosh(t)z_n,0,\ldots,0,\sinh(t)z_n\right).
	\end{align*}
	Now, $\hat{N}$ is a horizontal, time-like, unit normal vector field along $\hp$. An integral curve of $\hat{\xi}$ is
	\begin{equation} \label{integralxi} \alpha(s)=\hp\left(t+s/\vert z_n\vert, z\right)
	= \left( \sinh\left(t+s/\vert z_n\vert\right)z_n,z_1,\ldots,z_{n-1},\cosh\left(t+s/\vert z_n\vert\right)z_n\right).
	\end{equation}
	Let $\hat{A}$ be the shape operator associated with $\hat{N}$. Then, by \eqref{GaussEq}, 
	\begin{align*}
		&\hat{A}\hat{\xi}=-\hat{\nabla}_{\hat{\xi}}\hat{N} = - D_{\hat{\xi}}\hat{N} = - \left.\frac{d}{ds} \hat{N}_{\alpha(s)}\right \vert_{s=0} = - \left.\frac{d}{ds} \hat{N}_{\hp(t+s/\vert z_n\vert,z)}\right \vert_{s=0} \\
		& = - \left.
		\frac{d}{ds} \frac{i}{\vert z_n\vert} \left(\cosh\left(t+s/\vert z_n\vert\right)z_n,0,\ldots,0,\sinh\left(t+s/\vert z_n\vert\right)z_1\right) \right \vert_{s=0} \\
		&=\frac{-i}{\vert z_n\vert^2} \left(\sinh(t)z_n,0,\ldots,0, \cosh(t)z_n\right); \\
		& \langle \hat{A}\hat{\xi},\hat{\xi}\rangle=0. 
	\end{align*}
As before, $\psi(M)$ is a minimal ruled real hypersurface in $\cpn$. 
	
Let us study now the integral curves of $\hat{\xi}$. Note that $\langle \alpha',i\alpha\rangle=0$, which means that $\alpha$ is a horizontal curve. By \eqref{GaussEq}, 
\begin{align*}
		& F(s):= \hat{\nabla}_{\alpha'(s)}\alpha'(s) = D_{\alpha'(s)}\alpha'(s)+\langle \alpha'(s),\alpha'(s)\rangle \alpha(s) =
		\alpha''(s)-\alpha(s) \\
		& = \frac{1}{\vert z_n\vert^2}(\sinh\left(t+s/\vert z_n\vert\right),0,\ldots,0,\cosh(t+s/\vert z_n\vert)z_n)-\alpha(s) \\
		& = \left(\left(\frac{1}{\vert z_n\vert^2}-1\right)\sinh(t+s/\vert z_n\vert)z_n,-z_1,\ldots,-z_{n-1}, \left(\frac{1}{\vert z_n\vert^2}-1\right) \cosh(t+s/\vert z_n\vert)z_n\right); \\
		&  \langle F(s),F(s)\rangle =		\frac{1}{\vert z_n\vert^2}-1.
\end{align*}
When $\vert z_n\vert^2=1$, $F(s)$ is zero or light-like. The first case holds when $z=(0,\ldots,0,e^{ir})$, and the second one otherwise. When $\vert z_n\vert^2<1$, $F(s)$ is space-like. And when $\vert z_n\vert^2>1$, $F(s)$ is time-like. In addition, similar computations give $\langle F(s),\alpha(s)\rangle=\langle F(s),i\alpha(s)\rangle=0$, which show that $F(s)$ is a horizontal vector along $\alpha$. We need this to project it safely to $\cpn$. 
	
	\noindent $\bullet$ Case $\vert z_n\vert=1$. Then, $F(s)=(0,z_1,\ldots,z_{n-1},0)$, which is constant and either zero or light-like. The curve $\pi(\alpha(s))$ is the case $a)$ or $c)$ of Theorem \ref{TH1}. 
	
	\noindent $\bullet$ Case $\vert z_n\vert<1$. Then, $F(s)$ is space-like. We compute the Frenet system. 
\begin{align*}
		&\kappa_1:=\sqrt{\frac{1}{\vert z_n\vert^2}-1}>0, \quad F_1(s)=\alpha'(s),\ \varepsilon_1=-1, \\
		& F_2(s):=\frac{F(s)}{\kappa_1}
		=\left(\kappa_1\sinh(t+s/\vert z_n\vert)z_n,
		\frac{-z_1}{\kappa_1},\ldots,\frac{-z_{n-1}}{\kappa_1},\kappa_1(s)\cosh(t+s/\vert z_n\vert)z_n\right), \\
		&\varepsilon_2=\langle F_2(s),F_2(s)\rangle =1; \qquad 
		 \hat{\nabla}_{F_1(s)}F_1(s)=\kappa_1 F_2(s)=\varepsilon_2\kappa_1F_2(s).
	\end{align*}
Next, as $\langle F_2,\alpha'\rangle=0$, by \eqref{GaussEq},
	\begin{align*}
		\hat{\nabla}_{F_1(s)}F_2(s)=F_2'(s)+\langle F_1(s),F_2(s)\rangle\alpha(s)=\kappa_1F_1(s)=-\varepsilon_1\kappa_1F_1(s).
	\end{align*}
	This is a Frenet curve of order 2, with $F_1$ time-like, $F_2$ space-like, constant curvature $\kappa_1$ and 
	\begin{align*}
		&  \langle JF_1(s),F_2(s)\rangle = \frac{1}{\vert z_n\vert}\left\langle (i\cosh(t+s/\vert z_n\vert)z_n,0,\ldots,0,
		i\sinh(t+s/\vert z_n\vert)z_n),F_2(s)\right\rangle =0.
	\end{align*}
	Thus, if we project to $\cpn$, $\pi(\alpha)$ is a curve of type $b)$ in Theorem \ref{TH1}. 
	
	\noindent $\bullet$ Case $\vert z_n\vert>1$. Then, $F(s)$ is time-like. Similar computations give 
	\begin{align*}
		&\kappa_1:=\sqrt{1-\frac{1}{\vert z_n\vert^2}}>0, \quad F_1(s)=\alpha'(s),\ \varepsilon_1=-1, \\
		& F_2(s):=\frac{-F(s)}{\kappa_1}
		=\left(-\kappa_1\sinh(t+s/\vert z_n\vert)z_n,\frac{z_1}{\kappa_1},\ldots,\frac{z_{n-1}}{\kappa_1},-\kappa_1\cosh(t+s/\vert z_n\vert)z_n\right), \\
		&\langle F_2(s),F_2(s)\rangle = -1; 
		\qquad  \hat{\nabla}_{F_1(s)}F_1(s)=-\kappa_1 F_2(s)=\varepsilon_2\kappa_1F_2(s).
	\end{align*}
We also obtain $\langle JF_1(s),F_2(s)\rangle=0$. Next, as $\langle F_2,\alpha'\rangle=0$, by \eqref{GaussEq},
	\begin{align*}
		\hat{\nabla}_{F_1(s)}F_2(s)=F_2'(s)+\langle \alpha'(s),F_2(s)\rangle\alpha(s)=-\kappa_1F_1(s)=\varepsilon_1\kappa_1F_1(s).
	\end{align*}
	This is a Frenet curve of order 2, with $F_1$, $F_2$ time-like, constant curvature $\kappa_1$ and  totally real. 
	If we project to $\cpn$, $\pi(\alpha)$ is a curve of type $b)$ in Theorem \ref{TH1}. 
	
As in Example \ref{exam-first}, there exist different points lying in the same connected component of $\hat{M}$, such that the integral curves of $\hat{\xi}$ passing through them behave in a very different way. Projecting everything to $\cpn$, we have a  minimal ruled real hypersurface in $\cpn$ such that certain integral curves of $\xi$, starting at different points at the same connected component, behave as in any case of Theorem 2. 
\end{example}

\begin{example}\label{examp-second} \normalfont 
	We take $2\leq p\leq n-1$, $n\geq 3$. Consider $\Omega=\{z=(z_1,\ldots,z_n)\in \mathbb{S}_{2p}^{2n-1} : z_1\neq 0\}$. We define the map 
	\[\hp:\mathbb{R}\times \Omega\to \mathbb{S}^{2n+1}_{2p}, \quad \hp(t,z)=
	(\sin(t)z_1,\cos(t)z_1,z_2\ldots,z_{n}).\] 
We use the matrix $\mathbf{A}_{t}$ as in Example \ref{exam-first}. Define for each $t\in\mathbb{R}$, $\Omega^t:= \hp\left(\{t\}\times\Omega\right)=\Omega^t=\mathbf{A}_{t}(\hp(\{0\}\times \Omega))=\mathbf{A}_{t}(\Omega^0)$. 
Again, $\Omega^t$ is also totally geodesic, and so  $\pi(\Omega^t)\subset\cpn$ is a totally geodesic complex hyperplane. Then, we have the following ruled real hypersurface
	\[\psi: M=\RR\times \pi(\Omega) \to \cpn, \quad \psi(t,[z])=\pi(\hp(t,z)), \quad 
	\psi(M)=\bigcup_{t\in\mathbb{R} }\pi(\Omega^t).
	\] 
Similar computations provide 
	\begin{align*} \partial_t\hp & =(\cos(t)z_1,-\sin(t)z_1,0,\ldots,0), \\ 
		d\hp(0,X) & = (\sin(t)X_1,\cos(t)X_1,X_2,\ldots,X_{n}), \quad X\in T_z\Omega. 
	\end{align*}
	\begin{align*}
		\langle i\hp,d\hp(0,X)\rangle = \langle iz,X\rangle, & \quad \langle \partial_t\hp,d\hp(0,X)\rangle = 0, \\
		\langle d\hp(0,X),d\hp(0,Y)\rangle = \langle X,Y \rangle, & \quad 
		\langle \partial_t\hp,\partial_t\hp\rangle = -\vert z_1\vert^2<0. 
	\end{align*}
	
	\[ \hat{\xi}_{\hp(t,z)}=\frac{\partial_t\hp}{\vert z_1\vert}, \quad 
	\hat{N}_{\hp(t,z)} = i\hat{\xi}_{\hp(t,z)}=\frac{1}{\vert z_1\vert}\left(i \cos(t)z_1, -i \sin(t)z_1,0,\ldots,0\right).
	\]
In addition, $\hat{N}$ is a horizontal, time-like, unit normal vector field along $\hp$.	We point out that an integral curve of $\hat{\xi}$ is
\begin{equation} \label{integralxi} \alpha(s)=\hp\left(t+s/\vert z_1\vert, z\right)
	= \left(\sin\left(t+s/\vert z_1\vert\right)z_1,\cos\left(t+s/\vert z_1\vert\right)z_1,z_2\ldots,z_{n}\right).
\end{equation}
	Indeed, $\alpha(0)=\hp(t,z)$, $\alpha'(0)=
	\left. \frac{d}{ds}\hp\left(t+s/\vert z_1\vert,z\right) \right\vert_{s=0} =
	\frac{1}{\vert z_1\vert} \partial_t\hp (t,z)=\hat{\xi}_{\hp(t,z)}$, and 
	\[\alpha'(s)=\frac{1}{\vert z_1\vert}\left(\cos\left(t+s/\vert z_1\vert\right)z_1,-\sin\left(t+s/\vert z_1\vert\right)z_1,0,\ldots,0\right) = \hat{\xi}_{\hp(t+s/\vert z_1\vert,z)}= \hat{\xi}_{\alpha(s)}.
	\]
Let $\hat{A}$ be the shape operator associated with $\hat{N}$. Then, by \eqref{GaussEq}, 
	\begin{align*}
		&\hat{A}\hat{\xi}=-\hat{\nabla}_{\hat{\xi}}\hat{N} = - D_{\hat{\xi}}\hat{N} = - \left.\frac{d}{ds} \hat{N}_{\alpha(s)}\right \vert_{s=0} = - \left.\frac{d}{ds} \hat{N}_{\hp(t+s/\vert z_1\vert,z)}\right \vert_{s=0} \\
		& = - \left.
		\frac{d}{ds} \frac{i}{\vert z_1\vert}\left(\cos\left(t+s/\vert z_1\vert\right)z_1,-\sin\left(t+s/\vert z_1\vert\right)z_1,0,\ldots,0\right)\right \vert_{s=0} \\
		&=\frac{1}{\vert z_1\vert^2} \left(-i \sin(t)z_1,i\cos(t)z_1,0,\ldots,0\right); \\
		&\langle \hat{A}\hat{\xi},\hat{\xi}\rangle=0.
	\end{align*}
As before, $\psi(M)$ is a minimal ruled real hypersurface in $\cpn$.  
	
Let us study now the integral curves of $\hat{\xi}$. By \eqref{GaussEq}, 
	\begin{align*}
		& F(s):= \hat{\nabla}_{\alpha'(s)}\alpha'(s) = D_{\alpha'(s)}\alpha'(s)+\langle \alpha'(s),\alpha'(s)\rangle \alpha(s) =\alpha''(s)-\alpha(s) \\
		& = -\left( \left(\frac{1}{\vert z_1\vert^2}+1\right)\sin(t+s/\vert z_1\vert)z_1,
		\left(\frac{1}{\vert z_1\vert^2}+1\right)\cos(t+s/\vert z_1\vert )z_1,z_2,\ldots,z_n\right), \\
		& \langle F(s),F(s)\rangle=-1.
	\end{align*}
Given our computations so far, the Frenet system is 
	\begin{align*}
	&\kappa_1:=\sqrt{\frac{1}{\vert z_1\vert^2}+1}>0, \quad F_1(s)=\alpha'(s),\\
	& F_2(s):=\frac{-F(s)}{\kappa_1}
	= \left(-\kappa_1\sin(t+s/\vert z_1\vert)z_1,-\kappa_1 \cos(t+s/\vert z_1\vert)z_1,\frac{-z_2}{\kappa_1},\ldots,\frac{-z_{n}}{\kappa_1}\right), \\
	& \varepsilon_1=\langle \alpha',\alpha'\rangle=+1,\ \varepsilon_2=\langle F_2,F_2\rangle=-1, \quad 
	\hat{\nabla}_{F_1(s)}F_1(s)=-\kappa_1 F_2(s)=\varepsilon_2\kappa_1F_2(s). 
	\end{align*}
We also obtain $\langle JF_1(s),F_2(s)\rangle=0$. Next, since $\langle F_2,\alpha'\rangle=0$, by \eqref{GaussEq},
	\begin{align*}
	\hat{\nabla}_{F_1(s)}F_2(s)=F_2'(s)+\langle F_1(s),F_2(s)\rangle\alpha(s)=F_2'(s)=
	-\kappa_1F_1(s). 
	\end{align*}
If we project to $\cpn$, $\pi(\alpha)$ is a curve of type $b)$ in Theorem \ref{TH1}. 
	\end{example}

\end{document}